\documentclass{amsart}
\usepackage{amssymb}
\newtheorem{thm}{Theorem}[section]
\newtheorem{defin}[thm]{Definition}
\newtheorem{cor}[thm]{Corollary}
\newtheorem{lem}[thm]{Lemma}
\newtheorem{lemma-definition}[thm]{Lemma-Definition}
\newtheorem{proposition}[thm]{Proposition}
\theoremstyle{definition}
\theoremstyle{remark}

\numberwithin{equation}{section}

\def\Q{\mathbb{Q}}
\def\Z{\mathbb{Z}}
\def\N{\mathbb{N}}
\def\as#1{\renewcommand\arraystretch{#1}}
\def\dep{\operatorname{depth}}
\def\ff#1{\mathbb{F}_{#1}}
\def\gal{\operatorname{Gal}}

\def\lra{\longrightarrow}
\def\m{{\mathfrak m}}
\def\md#1{\ \mbox{\rm(mod }{#1})}
\def\oo{{\mathcal O}}

\def\qpb{\overline{K}}
\def\rd{\operatorname{red}}
\def\t{\theta}
\def\tq{\,\,|\,\,}
\def\ty{\mathbf{t}}
\def\zpx{\oo[x]}

\begin{document}
\title[Okutsu invariants and Newton polygons]{Okutsu invariants and Newton polygons}
\author[Gu\`ardia]{Jordi Gu\`ardia}
\address{Departament de Matem\`atica Aplicada IV, 
Escola Polit\`ecnica Superior d'Enginyera de Vilanova i la Geltr\'u, Av. V\'\i ctor Balaguer s/n. E-08800 Vilanova i la Geltr\'u, Catalonia}
\email{guardia@ma4.upc.edu}

\author[Montes]{\hbox{Jes\'us Montes}}
\address{Departament de Ci\`encies Econ\`omiques i Socials,
Facultat de Ci\`encies Socials,
Universitat Abat Oliba CEU,
Bellesguard 30, E-08022 Barce\-lona, Catalonia, Spain}
\email{montes3@uao.es}

\author[Nart]{\hbox{Enric Nart}}
\address{Departament de Matem\`{a}tiques,
         Universitat Aut\`{o}noma de Barcelo\-na,
         Edifici C, E-08193 Bellaterra, Barcelona, Catalonia, Spain}
\email{nart@mat.uab.cat}
\thanks{Partially supported by MTM2009-13060-C02-02 and MTM2009-10359 from the Spanish MEC}
\date{}
\keywords{local factorization, local field, Montes' algorithm, Montes approximation, Newton polygon, Okutsu frame}

\makeatletter
\@namedef{subjclassname@2010}{%
  \textup{2010} Mathematics Subject Classification}

\subjclass[2010]{Primary 11S15; Secondary 11Y05, 11Y40}

\begin{abstract}
Let $K$ be a local field of characteristic zero, $\oo$ its ring of integers, and $F(x)\in\zpx$ a monic irreducible polynomial. K. Okutsu attached to $F(x)$ certain  \emph{primitive divisor polynomials} $F_1(x),\dots, F_r(x)\in\zpx$, that are specially close to $F(x)$ with respect to their degree \cite{Ok}. In this paper we characterize the Okutsu families $F_1,\dots, F_r$, in terms of Newton polygons of higher order, and we derive some applications: closed formulas for certain Okutsu invariants, the discovery of new Okutsu invariants, or the construction of \emph{Montes approximations} to $F(x)$; these are monic irreducible polynomials sufficiently close to $F(x)$ to share all its Okutsu invariants. This perspective widens the scope of applications of Montes algorithm \cite{HN}, \cite{GMN}, which can be reinterpreted as a tool to compute the Okutsu polynomials and a Montes approxi\-mation, for each irreducible factor of a monic separable polynomial $f(x)\in\zpx$.
\end{abstract}

\subjclass[2010]{Primary 11S15; Secondary 11Y05, 11Y40}

\keywords{local factorization, local field, Montes algorithm, Montes approximation, Newton polygon, Okutsu frame}

\maketitle

\section{Introduction}
Let $K$ be a local field of characteristic zero, $\oo$ its ring of integers, $\m$ the maximal ideal of $\oo$, and $v\colon \qpb^*\lra \Q$, the canonical extension of the discrete valuation of $K$ to an algebraic closure of $K$.
Let $F(x)\in\zpx$ be a monic irreducible polynomial, $\t\in\qpb$ a root of $F(x)$, and $L=K(\t)$. K. Okutsu attached to $F(x)$ a family of monic irreducible 
polynomials, $F_1(x),\dots, F_r(x)\in\zpx$, 
called the \emph{primitive divisor polynomials} of $F(x)$
\cite{Ok}. Take $F_0(x)=1$. For each $1\le i\le r$, $\deg(F_i)$ is minimal among all monic irreducible polynomials $g(x)\in\zpx$ satisfying $v(g(\t))>v(F_{i-1}(\t))$, and $v(F_i(\t))$ is maximal among all polynomials having this mi\-nimal degree. 
Let us call the chain $[F_1,\dots,F_r]$ an \emph{Okutsu frame} of $F(x)$, and let $K_1,\dots, K_r$ be the respective extensions of $K$ 
determined by these polynomials. 
The polynomials $F_1,\dots, F_r$ are not uniquely determined, but many of their invariants, like 
the residual degrees $f(K_i/K)$ and the ramification indices $e(K_i/K)$, depend only on $F(x)$, and they are linked to some arithmetical invariants of the extension $L/K$ and its subextensions (Corollaries \ref{invariants} and \ref{tameincluded}).

In this paper we find a natural characterization of Okutsu frames in terms of Newton polygons of higher order \cite{HN}. More precisely, a family $[F_1,\dots, F_r]$ of monic irreducible polynomials in $\zpx$ is an Okutsu frame of $F(x)$ if and only if (Theorems \ref{niam} and \ref{main}):

\begin{enumerate}
 \item $F_1(x)$ is irreducible modulo $\m$ and divides $F(x)$ modulo $\m$.
\item $\deg F_1(x)<\cdots<\deg F_r(x)<\deg F(x)$.
\item For each $1\le i< r$, the Newton polygons of $i$-th order (with respect to $[F_1,\dots,F_i]$), $N_i(F)$ and $N_i(F_{i+1})$ are one-sided and they have the same negative slope.
\item For each $1\le i< r$, the residual polynomial of $i$-th order, $R_i(F_{i+1})$ is irreducible and, up to a multiplicative constant, $R_i(F)$ is a power of $R_i(F_{i+1})$. 
\item $N_r(F)$ is one-sided of negative slope and $R_r(F)$ is irreducible. 
\end{enumerate}

As a consequence, we get closed formulas for the invariants $v(F_i(\t))$ in terms of combinatorial data attached to these Newton polygons (Corollary \ref{vFi}), and we find new Okutsu invariants of $F(x)$ (Corollary \ref{newinvariants}). Newton polygons can be used too to construct \emph{Montes approximations} to $F(x)$; these are monic irreducible polynomials sufficiently close to $F(x)$ to share all its Okutsu invariants (Lemma \ref{sameframes}). In the tamely ramified case any Montes approximation to $F(x)$ generates the same extension of $K$ (Proposition \ref{sameextension}). 

Moreover, this characterization of the Okutsu frames provides the following reinterpretation of Montes algorithm \cite{HN}, \cite{GMN}: at the input of a monic se\-parable polynomial $f(x)\in\zpx$, the algorithm computes an Okutsu frame and a Montes approximation to each irreducible factor of $f(x)$. This widens the scope of applications of this algoritm, as a tool to compute the arithmetic information about the irreducible factors of $f(x)$, contained in their Okutsu invariants. For instance,  this perspective yields a measure of the precision of a Montes approximation (Lemma \ref{firstappr}), that makes it possible to slightly modify the algorithm in order to find approximations to the irreducible factors with prescribed precision (Section \ref{factorization}).

In another direction, the results of this paper open the door to a new 
cons\-truction of the prime ideals of the number field $M$ determined by a monic irreducible polynomial $f(x)$ with integer coefficients. For any prime number $p$, Montes algorithm computes Okutsu frames and Montes approxi\-mations to the different $p$-adic irreducible factors of $f(x)$. This can be interpreted as a  canonical paramete\-rization of the prime ideals of $M$ dividing $p$, in terms of Okutsu invariants that depend only on the polynomial $f(x)$. This parameterization is faithful enough to enable one to carry out the basic
tasks concerning ideals in number fields, without the necessity of
neither factorizing the discriminant of $f(x)$ nor constructing an integral basis of
the ring of integers of the number field. We hope to develop these ideas in a forthcoming paper \cite{GMN3}.

In the same vein, the divisors of a curve $C$ over  a finite field can be also parameterized in terms of Okutsu invariants that depend only on the defining equation of the curve. This enables one to compute the divisor of a function, or to construct a function with zeros and poles of a prescribed order at a finite number of places, avoiding the computation of integral bases of subrings of the function field of $C$. 

\section{Okutsu frames}
In this section we review and generalize Okutsu's results \cite{Ok}. The aim of the generalization is to facilitate the application of these results to the situation of section \ref{montes}.

Let $K$ be a local field of characteristic zero, $\oo$ its ring of integers, $\m$ the maximal ideal of $\oo$, $\pi\in\m$ a generator of $\m$, and $\ff{}$ the residue field. Let $\qpb$ be a fixed algebraic closure of $K$ and $v\colon \qpb^*\lra \Q$, the canonical extension of the discrete valuation of $K$.
We extend the valuation $v$ to the ring $\zpx$ in a natural way:
\begin{equation}\label{v1}
v\colon \zpx\lra \Z_{\ge0}\cup\{\infty\},\quad v(b_0+\cdots+b_rx^r):=\min\{v(b_j),\,0\le j\le r\}.
\end{equation}

For any finite extension $M$ of $K$ we denote by $\oo_M$ the  
ring of integers of $M$, and for any $\eta\in\qpb$ we denote  $\deg_M\eta:=[M(\eta)\colon M]$.

We fix throughout the section a monic irreducible polynomial $F(x)\in\zpx$ of degree $n$, a root $\t\in\qpb$ of $F(x)$, and the field $L=K(\t)$.

\subsection{Okutsu invariants}
Consider the two sequences, 
$$
\as{1.2}
\begin{array}{c}
m_0:=1\le m_1<m_2<\cdots<m_r<m_{r+1}=n,\\ 
\mu_0:=0<\mu_1<\mu_2<\cdots<\mu_r<\mu_{r+1}=\infty,
\end{array}
$$
the first one of positive integers, the second one of non-negative rational numbers (and infinity), defined in the following recurrent way for $i\ge 1$:
$$\as{1.2}
\begin{array}{l}
m_i:=\min\{\deg_K\eta\tq \eta\in\qpb,\ v(\t-\eta)>\mu_{i-1}\},\\
\mu_i:=\max\{v(\t-\eta)\tq \eta\in\qpb,\ \deg_K\eta=m_i\}.
\end{array}
$$
These numbers do not depend on the choice of $\t$ among the roots of $F(x)$. 

\begin{defin}\label{frame}
For $1\le i\le r$, choose $\alpha_i\in\qpb$ such that $\deg_K\alpha_i=m_i$, $v(\t-\alpha_i)=\mu_i$. Let  $F_i(x)\in\zpx$ be the minimal polynomial of $\alpha_i$, and  denote $K_i=K(\alpha_i)$.
The chain $[F_1,\dots,F_r]$ of monic irreducible polynomials in $\zpx$ is called an \emph{Okutsu frame} of $F(x)$. The length $r$ of an Okutsu frame is called the \emph{depth} of $F(x)$, and it will be denoted $\dep(F)$. 
\end{defin}

The polynomials $F_1,\dots,F_r$ that constitute an Okutsu frame of $F(x)$ are not uniquely determined. However, they are canonical in some sense, since many of their invariants (like the 
sequences $m_1<\cdots<m_r$ and $\mu_1<\cdots< \mu_r$) depend only on $F(x)$. More generally, an invariant of the family $F_1,\dots,F_r,F$ will be called an \emph{Okutsu invariant} if it does not depend on the choice of the Okutsu frame $[F_1,\dots,F_r]$. Thus, although the frame might be involved in their computation, they are actually invariants of $F(x)$. Corollaries \ref{invariants}, \ref{more} and \ref{newinvariants} present more Okutsu invariants of $F(x)$. 

\begin{lem}\label{v0}
For any $h(x)\in\zpx$, we have $v(h(\t))=0$ if and only if $h(x)$ is relatively prime to $F(x)$ modulo $\m$.
\end{lem}

\begin{proof}
If there exist $a(x),b(x)\in\zpx$ such that $a(x)F(x)+b(x)h(x)\in 1+\m[x]$, then clearly $v(h(\t))=0$. 

By Hensel lemma, $F(x)\equiv F_1(x)^\ell\md{\m}$, for some $F_1(x)\in\zpx$ which is irreducible modulo $\m$. Clearly, $v(F_1(\t))>0$, and if $F_1(x)$ divides $h(x)$ modulo $\m$, we have $v(h(\t))>0$ too.
\end{proof}

\begin{cor}\label{phi1}\mbox{\null}
\begin{enumerate}
\item[(i)] A monic irreducible polynomial $F(x)\in\zpx$ has depth zero if and only if $F(x)$ is irreducible modulo $\m$. In this case, the Okutsu frames of $F(x)$ are all empty.
\item[(ii)] If $[F_1,\dots,F_r]$ is an Okutsu frame of $F(x)$, then $F_1(x)$ is irreducible modulo $\m$, $m_1|n$, and $F(X)\equiv F_1(x)^{n/m_1}\md{\m}$.\qed
\end{enumerate}
\end{cor}

\begin{lem}\label{recurrence}
Let $[F_1,\dots,F_r]$ be an Okutsu frame of $F(x)$. For some index, $1\le i\le r+1$, let $\alpha\in\qpb$ be an algebraic integer satisfying: $\deg_K \alpha=m_i$, $v(\t-\alpha)>\mu_{i-1}$, and let $G(x)$ be the minimal polynomial of $\alpha$ over $K$.
Then, $[F_1,\dots,F_{i-1}]$ is an Okutsu frame of $G(x)$. 
\end{lem}

\begin{proof}
Since the sequence $\mu_j=v(\t-\alpha_j)$ is strictly increasing, 
$$
v(\alpha-\alpha_j)=\min\{v(\alpha-\t),v(\t-\alpha_j)\}=\mu_j,\quad \forall\, 1\le j<i.
$$ 
Now, for all $1\le j<i$, and all $\eta\in\qpb$ we have:
$$v(\alpha-\eta)>\mu_{j-1}\,\Longrightarrow\,
v(\t-\eta)>\mu_{j-1}\,\Longrightarrow\, 
\deg_K\eta \ge m_j,
$$$$
\deg_K\eta =m_j\,\Longrightarrow\, 
v(\t-\eta)\le \mu_j\,\Longrightarrow\,
v(\alpha-\eta)\le \mu_j.
$$
\end{proof}

\begin{cor}\label{recFi}
Let $[F_1,\dots,F_r]$ be an Okutsu frame of $F(x)$. Then, for all $1\le i\le r$, $[F_1,\dots,F_{i-1}]$ is an Okutsu frame of $F_i(x)$. In parti\-cular, $\dep(F_i)=i-1$.\qed
\end{cor}

\begin{lem}\label{taylor}
For some $1\le i\le r+1$, let $\alpha,\eta\in\qpb$ be algebraic integers satisfying: 
$$v(\t-\alpha)>\mu_{i-1},\quad v(\t-\eta)>\mu_{i-1}.
$$
Then, for any non-zero polynomial $g(x)\in K[x]$ of degree less than $m_i$, we have $v(g(\eta)-g(\alpha))>v(g(\alpha))$. 

In particular, if $\deg_K\alpha=m_i$, then $e(K(\alpha)/K)$ divides  $e(K(\eta)/K)$. 

\end{lem}

\begin{proof}
By dividing $g(x)$ by its leading coefficient we can suppose that $g(x)$ is monic.
If $g(x)$ has degree zero, the statement is obvious. Suppose $g(x)$ has positive degree $s<m_i$, with roots $\rho_1,\dots,\rho_s$ in $\qpb$. By the minimality of $m_i$, we have necessarily $v(\t-\rho_j)\le\mu_{i-1}$ for all $j$. Hence,
\begin{equation}\label{aux}
v(\rho_j-\alpha)=\min\{v(\t-\rho_j),v(\t-\alpha)\}\le\mu_{i-1}< v(\eta-\alpha). 
\end{equation}
By Taylor's development,
$$
g(\eta)-g(\alpha)=g'(\alpha)(\eta-\alpha)+\cdots+\dfrac{g^{(k)}(\alpha)}{k!}(\eta-\alpha)^k+\cdots
$$
Now, (\ref{aux}) and the formula
$$
\dfrac{g^{(k)}(x)}{k!}=\sum_{\mbox{\begin{tiny}$\begin{array}{c}S\subseteq\{1,\dots,s\},\\\#S=k\end{array}$\end{tiny}}}\prod_{j\not\in S}(x-\rho_j),
$$
show that: $v\left(\dfrac{g^{(k)}(\alpha)}{k!}(\eta-\alpha)^k\right)>v(g(\alpha))$, for all $k\ge 1$. 

If $\deg_K\alpha=m_i$, any $u\in K(\alpha)$ with $v(u)=1/e(K(\alpha)/K)$ can be expressed as $u=g(\alpha)$, for some $g(x)\in K[x]$ of degree less than $m_i$. Hence, from $v(g(\eta))=v(g(\alpha))=1/e(K(\alpha)/K)$, we deduce that $e(K(\alpha)/K)$ divides  $e(K(\eta)/K)$. 
\end{proof}

The algebraic integers $\alpha_1,\dots,\alpha_r$ are close to $\t$ with regard to their degree, but the fields $K_1,\dots,K_r$ are not necessarily subfields of $L$. However, the next proposition, which is a generalization of \cite[II,Prop.4]{Ok},  
shows that the maximal tamely ramified subextensions of $K_1/K,\dots,K_r/K$ are always included in $L$.

\begin{proposition}\label{tame}
For some $1\le i\le r+1$, suppose that $\alpha\in\qpb$ satisfies: 
$$\deg_K\alpha=m_i,\quad v(\t-\alpha)>\mu_{i-1},
$$and let $N=K(\alpha)$. Let $M/K$ be any finite Galois extension containing $L$ and $N$, and let  $G=\gal(M/K)$. Consider the subgroups
$$
H_i=\{\sigma\in G\tq v(\t-\sigma(\t))> \mu_{i-1}\}\supseteq
H'_i=\{\sigma\in G\tq v(\t-\sigma(\t))\ge \mu_i\},
$$  
and let $M^{H_i}\subseteq M^{H'_i}\subseteq M$ be the respective fixed fields. Finally, let $V$ be the maximal tamely ramified subextension of $N/K$. Then, $V\subseteq M^{H_i}\subseteq L\cap N$. If moreover $v(\t-\alpha)=\mu_i$, then $V\subseteq M^{H_i}\subseteq M^{H'_i}\subseteq L\cap N$.
\end{proposition}

\begin{proof}
In order to show that $M^{H_i}\subseteq K(\t)\cap K(\alpha)$, we need to check that all $\sigma\in G$ that leave $\t$ or $\alpha$ invariant lie in $H_i$. This is obvious for the automorphisms leaving $\t$ invariant; for the $\sigma\in G$ such that $\sigma(\alpha)=\alpha$, we have $v(\sigma(\t)-\alpha)=v(\sigma(\t)-\sigma(\alpha))>\mu_{i-1}$, so that $$v(\t-\sigma(\t))\ge \min\{v(\t-\alpha), v(\alpha-\sigma(\t))\}>\mu_{i-1}.$$  The same argument shows that $M^{H'_i}\subseteq L\cap N$, if $v(\t-\alpha)=\mu_i$.\medskip

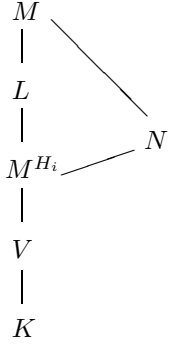
\begin{center}
\setlength{\unitlength}{4.4mm}
\begin{picture}(4,10)
\put(0,9.6){$M$}\put(0,7.2){$L$}\put(-.2,4.8){$M^{H_i}$}\put(0,2.4){$V$}\put(0,0){$K$}
\put(0.3,1){\line(0,1){1}}\put(0.3,3.5){\line(0,1){1}}
\put(0.3,5.9){\line(0,1){1}}\put(0.3,8.3){\line(0,1){1}}
\put(4,5.7){$N$}
\put(1.5,4.9){\line(3,1){2.2}}
\put(1.2,9.6){\line(1,-1){2.9}}
\end{picture}
\end{center}

By the basic properties of tamely ramified extensions \cite[Ch.5,\S2]{Nark}, in order to check that $V\subseteq M^{H_i}$, it is sufficient to show that
$$
v\left(\dfrac{\sigma(u)}{u}-1\right)>0,\quad \forall u\in N^*,\ \forall \sigma\in H_i.
$$
Clearly, $v(\t-\sigma(\alpha))\ge \min\{v(\t-\sigma(\t)),v(\sigma(\t)-\sigma(\alpha))\}>\mu_{i-1}$, for all $\sigma\in H_i$. Finally, any $u\in N^*$ can be written as $u=g(\alpha)$, for some $g(x)\in K[x]$ of degree less than $m_i$.
By Lemma \ref{taylor} applied to $\eta=\sigma(\alpha)$, we have $$v(\sigma(u)-u)=v(g(\sigma(\alpha))-g(\alpha))>v(g(\alpha))=v(u),$$ as desired.  
\end{proof}

The following two corollaries generalize \cite[II,Prop.6, Cors.1,2]{Ok}.

\begin{cor}\label{invariants}
Let  $[F_1,\dots,F_r]$ be an Okutsu frame of $F(x)$. Then,
the numbers $e(K_i/K)$, $f(K_i/K)$, for $1\le i\le r$, do not depend on the chosen Okutsu frame. Moreover,
$$e(K_1/K)\,|\,\cdots\,|\,e(K_r/K)\,|\,e(L/K),\quad f(K_1/K)\,|\,\cdots\,|\,f(K_r/K)\,|\,f(L/K).$$ 
In particular,
$m_1\,|\,\cdots\,|\,m_r\,|\,m_{r+1}=\deg(F)$.
\end{cor}

\begin{proof}
For all $1\le i\le r$, suppose that $\eta\in\qpb$ satisfies $\deg_K\eta=m_i$ and $v(\t-\eta)=\mu_i$. By Lemma \ref{taylor}, 
$e(K_i/K)=e(K(\eta)/K)$, because these two numbers divide one to each other. Hence, $f(K_i/K)=m_i/e(K_i/K)=m_i/e(K(\eta)/K)=f(K(\eta)/K)$. 

By Proposition \ref{tame}, we have $f(K_r/K)\,|\,f(L/K)$. By Lemma \ref{taylor}, applied to $i=r$, $\alpha=\alpha_r$ and $\eta=\t$, we get $e(K_r/K)\,|\,e(L/K)$. The rest of the statements follow from Corollary \ref{recFi} by a recurrent argument.
\end{proof}

\begin{cor}\label{tameincluded}
Suppose $L/K$ tamely ramified. Let $[F_1,\dots,F_r]$ be an Okutsu frame of $F(x)$, and let $K_{r+1}:=L$. Then,
\begin{enumerate}
\item[(i)] $K_i=M^{H_i}=M^{H'_i}$, for all $1\le i\le r+1$. Hence, $K_1\subseteq \cdots\subseteq K_r\subseteq L$.
\item[(ii)] $\{v(\t-\sigma(\t))\tq \sigma\in G\}=\left\{\begin{array}{ll}\{\mu_1,\dots,\mu_r,\mu_{r+1}\}, & \mbox{ if }m_1=1,\\\{\mu_0,\mu_1,\dots,\mu_r,\mu_{r+1}\}, & \mbox{ if }m_1>1. \end{array}
\right.$ 

In particular, $\mu_r$ is Krasner radius of $F(x)$: $$\mu_r=\max\{v(\t-\t')\tq F(\t')=0, \,\t'\ne\t\}.$$
\end{enumerate}
\end{cor}

\begin{proof}
Denote $\alpha_{r+1}:=\t$. By Corollary \ref{invariants}, if $L/K$ is tamely ramified, then $K_i/K$ is tamely ramified, for all $1\le i\le r+1$. Proposition \ref{tame}, applied to each $\alpha_i$, shows that $K_i=M^{H_i}=M^{H'_i}$, and hence,
$H_i=H'_i$, for all $1\le i\le r+1$.  Moreover, for all $1\le i\le r$ we have $H_{i+1}\subsetneq H_i=H'_i$, because $[K_{i+1}\colon K_i]>1$; hence, there is some $\sigma\in H_i$ with $v(\t-\sigma(\t))=\mu_i$. Finally, there is no
$\sigma\in G$ with $v(\t-\sigma(\t))=0$, if and only if $H_1=G$, or equivalently, $K_1=K$.  
\end{proof}

\subsection{Okutsu frames and divisor polynomials}
\begin{proposition}{\cite[I,Prop.1]{Ok}}\label{DP}
 For any integer $0\le m<n$ there exists a monic polynomial $g_m(x)\in\zpx$ of degree $m$ such that 
\begin{equation}\label{vDP}
 v(g_m(\t))\ge v(g(\t))-v(g(x)),
\end{equation}
for all polynomials $g(x)\in\zpx$ of degree $m$. 
\end{proposition}

\begin{proof}
Let $g(x)\in\zpx$ be a monic polynomial of degree $m$.  Let $\oo'\subseteq \oo_L$ be the $\oo$-module generated by $\t$ and $g(\t)/\pi^{\lfloor v(g(\t))\rfloor}$. Clearly, 
$$
\lfloor v(g(\t))\rfloor=\ell\left(\oo'/\oo[\t]\right)\le \ell\left(\oo_L/\oo[\t]\right)<\infty,
$$where $\ell$ denotes the length as an $\oo$-module.
Since $v$ restricted to $L$ is discrete, $v(g(\t))$ takes only a finite number of values; hence, there exists a monic $g_m(x)\in\zpx$ of degree $m$ such that $v(g_m(\t))\ge v(g(\t))$ for all monic $g(x)\in\zpx$ of degree $m$.
In order to check (\ref{vDP}) it is sufficient to show that $v(g_m(\t))\ge v(g(\t))$ for any $g(x)\in\zpx$ of degree $m$ such that $v(g(x))=0$. Let us prove this by induction on $m$. For $m=0$ we have $g_m(x)=1$ and the statement is obvious. Suppose $m>0$. If $a\in\oo$ is the leading coefficient of $g(x)$, we write
$$
g(x)=ag_m(x)+r(x), \quad m':=\deg r(x)<m.
$$ 
If $v(a)=0$, then $v(g_m(\t))\ge v(a^{-1}g(\t))=v(g(\t))$, by the construction of $g_m(x)$. If $v(a)>0$, then $v(r(x))=0$ and by the induction hypothesis
$$
v(g_m(\t)) \ge v(\t^{m-m'}g_{m'}(\t)) \ge v(g_{m'}(\t)) \ge v(r(\t)).
$$ 
Thus, $v(r(\t))<v(ag_m(\t))$, so that $v(g(\t))=v(r(\t))\le v(g_m(\t))$.
\end{proof}

\begin{defin}
Clearly, (\ref{vDP}) does not depend on the choice of $\t$ among the roots of $F(x)$. We call $g_m(x)$ a \emph{divisor polynomial} of degree $m$ of $F(x)$. 
\end{defin}

\begin{lem}\label{comparison}
Let $[F_1,\dots,F_r]$ be an Okutsu frame of $F(x)$. Let $h(x)\in\zpx$ be a monic irreducible polynomial of degree $m$, and let
$$
\delta_F(h)=\max\{v(\t-\beta)\tq \beta\in\qpb\mbox{ is a root of }h(x)\}.
$$Then, for any $1\le i\le r$,
\begin{enumerate}
\item[(i)] $\delta_F(h)<\mu_i\ \Longrightarrow\ v(h(\t))<(m/m_i)v(F_i(\t))$,
\item[(ii)] $\delta_F(h)=\mu_i\ \Longrightarrow\ v(h(\t))=(m/m_i)v(F_i(\t))$,
\item[(iii)] $\delta_F(h)>\mu_i\ \Longrightarrow\ v(h(\t))>(m/m_i)v(F_i(\t))$.
\end{enumerate}
\end{lem}

\begin{proof}
Take a finite Galois extension $M/K$ containing $L$, $K_i$ and the roots of $h(x)$. Denote $G:=\gal(M/K)$. Choose a root $\beta$ of $h(x)$ such that $v(\t-\beta)=\delta_F(h)\ge v(\t-\sigma(\beta))$, for all $\sigma\in G$. This implies 
\begin{equation}\label{bound}
v(\t-\sigma(\t))\ge\min\{v(\t-\sigma(\beta)),v(\sigma(\beta)-\sigma(\t))\}= v(\t-\sigma(\beta)),  
\end{equation}
for all $\sigma\in G$. Suppose $\delta_F(h)\le \mu_i$. We claim that,
\begin{equation}\label{ab}
v(\t-\sigma(\beta)) \le v(\t-\sigma(\alpha_i)),\quad\forall \,\sigma\in G.
\end{equation}
In fact, by the maximality of $\mu_i$, we have $v(\t-\sigma(\alpha_i))\le \mu_i$, for all $\sigma\in G$. Now, if  $v(\t-\sigma(\alpha_i))=\mu_i$, then the inequality (\ref{ab}) is obvious; on the other hand, if  $v(\t-\sigma(\alpha_i))<\mu_i$, we have by (\ref{bound}) 
$$
v(\t-\sigma(\beta))\le v(\t-\sigma(\t))=
\min\{v(\t-\sigma(\alpha_i)),v(\sigma(\alpha_i)-\sigma(\t))\}=v(\t-\sigma(\alpha_i)).
$$Therefore, (\ref{ab}) shows that
\begin{equation}\label{boundirred}
\dfrac{\#G}{m}v(h(\t))=\sum_{\sigma\in G}v(\t-\sigma(\beta))\le \sum_{\sigma\in G}v(\t-\sigma(\alpha_i))=\dfrac{\#G}{m_i}v(F_i(\t)).
\end{equation}
If  $\delta_F(h)< \mu_i$, then the inequality in (\ref{ab}) is strict, at least for $\sigma=1$; hence, the inequality in (\ref{boundirred}) is strict too. This proves item (i). 

Suppose now $\delta_F(h)\ge\mu_i$. Then, $v(\alpha_i-\beta)\ge\min\{v(\alpha_i-\t),v(\t-\beta)\}=\mu_i$, and we have directly, for all $\sigma\in G$: 
\begin{equation}\label{ab2}
v(\t-\sigma(\beta))\ge \min\{v(\t-\sigma(\alpha_i)),v(\sigma(\alpha_i)-\sigma(\beta))\}=v(\t-\sigma(\alpha_i)).
\end{equation} Therefore,
\begin{equation}\label{boundirred2}
\dfrac{\#G}{m}v(h(\t))=\sum_{\sigma\in G}v(\t-\sigma(\beta))\ge\sum_{\sigma\in G}v(\t-\sigma(\alpha_i))=\dfrac{\#G}{m_i}v(F_i(\t)).
\end{equation} The inequalities (\ref{boundirred}) and (\ref{boundirred2}) prove item (ii). Finally, if  $\delta_F(h)>\mu_i$, the inequality (\ref{ab2}) is strict at least for $\sigma=1$, and the inequality in (\ref{boundirred2}) is strict too. This proves item (iii). 
\end{proof}

The next result is a generalization of \cite[II,Prop.2]{Ok}.

\begin{thm}\label{PDP}
Let $m_1<\dots<m_r<\deg(F)$ be the Okutsu degrees of $F(x)$, and let  $F_1(x),\dots,F_r(x)\in\zpx$ be a family of 
monic polynomials with $\deg(F_i(x))=m_i$, for all $1\le i\le r$. Then, $[F_1,\dots,F_r]$ is an Okutsu frame of $F(x)$ if and only if each $F_i(x)$ is a divisor polynomial of $F(x)$ of degree $m_i$.
\end{thm}

\begin{proof}
Suppose that $[F_1,\dots,F_r]$ is an Okutsu frame of $F(x)$. For any index $1\le i\le r$, let
 $g(x)\in\zpx$ be a monic polynomial of degree $m_i$, and let
$g(x)=h_1(x)\cdots h_s(x)$ with $h_j(x)\in\zpx$ monic and irreducible. By the maximality of $\mu_i$, we have $\delta_F(h_j)\le \mu_i$, for all $j$. Lemma \ref{comparison} shows that $v(h_j(\t))\le (\deg(h_j)/m_i)v(F_i(\t))$, for all $j$; hence,
$$
v(g(\t))=\sum_jv(h_j(\t))\le \sum_j (\deg h_j/m_i) v(F_i(\t))=v(F_i(\t)).
$$Along the proof of Proposition \ref{DP} we saw that this property, for all monic $g(x)$ of degree $m_i$, implies already that $F_i(x)$ is a divisor polynomial of degree $m_i$.

Conversely, suppose that each $F_i(x)$ is a divisor polynomial of $F(x)$ of degree $m_i$, and let $[F'_1,\dots,F'_r]$ be an Okutsu frame of $F(x)$. Fix an index  $1\le i\le r$; clearly, $v(F_i(\t))=v(F'_i(\t))$, since $F_i(x)$ and $F'_i(x)$ are both divisor polynomials of degree $m_i$.   Let us see first that  $F_i(x)$ is necessarily irreducible. In fact, suppose $F_i(x)=h_1(x)\cdots h_s(x)$, with all $h_j(x)$ monic irreducible polynomials of degree less than $m_i$. Then, $\delta_F(h_j)\le \mu_{i-1}<\mu_i$, for all $j$, by the minimality of $m_i$. By Lemma \ref{comparison}, $v(h_j(\t))<(\deg(h_j)/m_i)v(F'_i(\t))$, for all $j$, and this implies $v(F_i(\t))<v(F'_i(\t))$, in contradiction with our assumption. Once we know that $F_i(x)$ is irreducible, 
 Lemma \ref{comparison} shows that $\delta_F(F_i)=\mu_i$. Thus, $[F_1,\dots,F_r]$ is an Okutsu frame of $F(x)$.  
\end{proof}

\begin{cor}\label{more}
The rational numbers $0<v(F_1(\t))<\cdots <v(F_r(\t))$, depend only on $F(x)$.\qed
\end{cor}

In Corollary \ref{vFi} below, we shall show how to compute these invariants solely in terms of the Okutsu frame $[F_1,\dots,F_r]$.

Okutsu refers to $F_1(x),\dots,F_r(x)$ as the \emph{primitive divisor polynomials} of $F(x)$, because by multiplying them in an adequate way one obtains the divisor polynomials of all degrees $0\le m<n$. 

\begin{thm}{\cite[II,Thm.1]{Ok}}\label{allm}
Let $[F_1,\dots,F_r]$ be an Okutsu frame of $F(x)$. Take $F_0(x)=x$. Let $0<m<n$, and write it (in a unique form) as
$$
m=\sum_{i=0}^ra_im_i,\quad \mbox{ with }\ 0\le a_i<\dfrac{m_{i+1}}{m_i}. 
$$ 
Then, $g_m(x):=\prod_{i=0}^rF_i(x)^{a_i}$ is a divisor polynomial of degree $m$ of $F(x)$.
\end{thm}

\begin{proof}
As we saw in the proof of Proposition \ref{DP}, it is sufficient to show that $v(g_m(\t))\ge v(g(\t))$ for all $g(x)\in\zpx$ monic of degree $m$. 
 Let us prove this by induction on $m$. 

If $m<m_1$, then $g_m(x)=x^m$ and $v(g_m(\t))=0=v(g(\t))$, for all $g(x)$ monic of degree $m$, by Lemma \ref{v0} and Corollary \ref{phi1}. If $m=m_1$ then $g_m(x)=F_1(x)$, which is a divisor polynomial of degree $m_1$ of $F(x)$ by Theorem \ref{PDP}. Thus, the theorem is proven for all $0<m\le m_1$.

Let $m_j\le m<m_{j+1}$, for some $1\le j\le r$, and suppose that the theorem is true for all degrees less than $m$.   
We claim that $v(\t-\eta)\le v(\alpha_j-\eta)$, for any root $\eta\in\qpb$ of  any monic polynomial $g(x)\in\zpx$ of degree $m$. In fact, $v(\t-\eta)\le\mu_j$, by the minimality of $m_{j+1}$, and
$$
\begin{array}{l}
v(\t-\eta)<\mu_j\,\Longrightarrow\,v(\alpha_j-\eta)=\min\{v(\t-\eta),v(\t-\alpha_j)\}=v(\t-\eta),\\
v(\t-\eta)=\mu_j\,\Longrightarrow\,v(\alpha_j-\eta)\ge\min\{v(\t-\eta),v(\t-\alpha_j)\}=\mu_j.
\end{array}
$$
Hence, $v(g(\t))\le v(g(\alpha_j))$. Consider now
$$
g(x)=F_j(x)q(x)+r(x),\quad \deg (r(x))<m_j.
$$
Clearly, $v(g(\t))\le v(g(\alpha_j))=v(r(\alpha_j))=v(r(\t))$,
the last equality by Lemma \ref{taylor} applied to $\alpha=\alpha_j$, $\eta=\t$. Since $\deg(q(x))=m-m_j$, we have
$$v(q(\t))\le v\left(F_j(\t)^{a_j-1}\prod_{k=0}^{j-1}F_k(\t)^{a_k}\right),$$ by the induction hypothesis; 
hence
 $$v\left(\prod_{i=0}^rF_i(\t)^{a_i}\right) \ge v(F_j(\t)q(\t))\ge 
\min\{ v(g(\t)),v(r(\t))\}=v(g(\t)).$$
\end{proof}

\section{Okutsu frames and Newton polygons of higher order} 
In this section we study more properties of Okutsu frames in connection with the theory of Newton polygons of higher order. These polygons were introduced in \cite{montes}, and revised in \cite{HN} (HN standing for ``higher Newton"). 

We keep dealing with a local field $K$ of characteristic zero, and we keep the notations $\oo$, $\m$, $\pi$, $\ff{}$ and $v$, of the previous section. Also, we keep fixing a monic irreducible polynomial $F(x)\in\zpx$ of degree $n$, a root $\t\in\qpb$ of $F(x)$, and the field $L=K(\t)$.

The aim of this section is to characterize when a chain $[F_1,\dots,F_r]$ of monic irreducible polynomials is an Okutsu frame of $F(x)$ in terms of invariants linked to certain Newton polygons (Theorems \ref{niam} and \ref{main}). As a consequence, we give a closed formula for the invariants $v(F_i(\t))$ solely in terms of the Okutsu frame (Corollary \ref{vFi}), and we find
new Okutsu invariants of $F(x)$ (Corollary \ref{newinvariants}). \medskip

\noindent{\bf Notation. }{\it
Let $\mathcal{F}$ be a field and $\varphi(y),\,\psi(y)\in \mathcal{F}[y]$ two polynomials. We write $\varphi(y)\sim \psi(y)$ to indicate that  there exists a constant $c\in\mathcal{F}^*$ such that $\varphi(y)=c\psi(y)$}.\medskip

\subsection{Newton polygons of higher order}
Let us briefly review Newton polygons of higher order \cite[Sects.1,2]{HN}. We denote by $v_1$ the discrete va\-luation on $K(x)$ determined by the natural extension of $v$ to polynomials given in (\ref{v1}). Consider the $0$-th \emph{residual polynomial} operator
$$
R_0\colon \zpx\lra \ff{}[y],\quad g(x)\mapsto \rd\left(g(y)/\pi^{v_1(g)}\right), 
$$
where, $\rd\colon \oo[y]\to \ff{}[y]$, is the natural reduction map.
A \emph{type of order zero}, $\ty=\psi_0(y)$, is just a monic irreducible polynomial $\psi_0(y)\in\ff{}[y]$. A \emph{representative} of $\ty$ is any monic polynomial $\phi_1(x)\in\zpx$ such that $R_0(\phi_1)=\psi_0$. The pair $(\phi_1,v_1)$ can be used to attach a Newton polygon to any nonzero polynomial $g(x)\in K[x]$. If $g(x)=\sum_{i\ge0}a_i(x)\phi_1(x)^i$ is the $\phi_1$-adic development of $g(x)$, then 
$N_1(g):=N_{\phi_1,v_1}(g)$ is the lower convex envelope of the set of points of the plane with coordinates $(i,v_1(a_i(x)\phi_1(x)^i))$ \cite[Sec.1]{HN}. 

Let $\lambda_1\in\Q^-$ be a negative rational number, $\lambda_1=-h_1/e_1$, with $h_1,e_1$ po\-sitive coprime integers. The triple $(\phi_1,v_1,\lambda_1)$ determines a discrete valuation $v_2$ on $K(x)$, constructed as follows: for any  nonzero polynomial $g(x)\in\zpx$, take a line of slope $\lambda_1$ far below $N_1(g)$ and let it shift upwards till it touches the polygon for the first time; if $H$ is the ordinate at the origin of this line, then $v_2(g(x))=e_1 H$ by definition. Also, the triple $(\phi_1,v_1,\lambda_1)$ determines a \emph{residual polynomial}  operator
$$
R_1:=R_{\phi_1,v_1,\lambda_1}\colon \zpx\lra \ff1[y],\quad \ff1:=\zpx/(\m[x],\phi_1(x))
$$
which is a kind of reduction of first order of $g(x)$ \cite[Def.1.9]{HN}. 

Let $\psi_1(y)\in\ff1[y]$ be a monic irreducible polynomial, $\psi_1(y)\ne y$. The triple $\ty=(\phi_1(x);\lambda_1,\psi_1(y))$ is called a \emph{type of order one}. Given any such type, one can compute a representative of $\ty$; that is, a monic polynomial $\phi_2(x)\in\zpx$ satisfying
$R_1(\phi_2)(y)\sim\psi_1(y)$. 

Now we may start over and repeat all constructions in order two. The pair $(\phi_2,v_2)$
can be used to attach a Newton polygon $N_2(g):=N_{\phi_2,v_2}(g)$ to any nonzero polynomial $g(x)\in K[x]$. This polygon is constructed from the $\phi_2$-adic development $g(x)=\sum_{i\ge0}b_i(x)\phi_2(x)^i$, as the lower convex envelope of the set of points of the plane with coordinates $(i,v_2(b_i(x)\phi_2(x)^i))$. 
For any negative rational number $\lambda_2$, the triple $(\phi_2,v_2,\lambda_2)$ determines a discrete valuation $v_3$ on $K(x)$ and a  \emph{residual polynomial} operator
$$
R_2:=R_{\phi_2,v_2,\lambda_2}\colon \zpx\lra \ff2[y],\quad \ff2:=\ff1[y]/(\psi_1(y)).
$$

The iteration of this procedure leads to the concept of \emph{type of order $r$}. A type of order $r\ge 1$ is a chain:
$$
\ty=(\phi_1(x);\lambda_1,\phi_2(x);\cdots;\lambda_{r-1},\phi_r(x);\lambda_r,\psi_r(y)),
$$
where $\phi_1(x),\dots,\phi_r(x)$ are monic irreducible polynomials in $\zpx$, $\lambda_1,\dots,\lambda_r$ are negative rational numbers, and $\psi_r(y)$ is a polynomial over certain finite field $\ff{r}$ (to be specified below), that satisfy the following recursive properties:
\begin{enumerate}
\item $\phi_1(x)$ is irreducible modulo $\m$.  We denote $\psi_0(y):=R_0(\phi_1)(y)\in \ff{}[y]$, and we define $\ff1=\ff{}[y]/(\psi_0(y))$.
\item For all $1\le i<r$, $N_i(\phi_{i+1}):=N_{\phi_i,v_i}(\phi_{i+1})$ is one-sided of slope $\lambda_i$, and $R_i(\phi_{i+1})(y):=R_{\phi_i,v_i,\lambda_i}(\phi_{i+1})(y)\sim \psi_i(y)$, for some monic irreducible polynomial $\psi_i(y)\in \ff{i}[y]$. We define $\ff{i+1}=\ff{i}[y]/(\psi_i(y))$.
\item $\psi_r(y)\in\ff{r}[y]$ is a monic irreducible polynomial, $\psi_r(y)\ne y$. 
\end{enumerate}

We attach to any type $\ty$ of order $r$ certain invariants: $e_1,\dots,e_r$, $h_1,\dots,h_r$, $f_0,f_1,\dots,f_r$.
They are defined as follows: $\lambda_i=-h_i/e_i$, with $e_i,h_i$ positive coprime integers, and $f_i=\deg(\psi_i(y))$. In the sequel we shall extensively use these numerical invariants without any mention to the underlying type $\ty$, that will be usually implicit in the context.

For any nonzero $g(x)\in K[x]$, we denote by $N_i^-(g)$ the union of the sides of negative slope of $N_i(g)$. We say that $N_i^-(g)$ is the \emph{principal part} of $N_i(g)$. The \emph{length} of either of these two polygons is by definition the length of its projection to the horizontal axis.

At each order $1,2,\dots,r$, three fundamentals theorems provide a far-reaching generalization of Hensel lemma, that had been worked out by Ore in order one \cite{ore1}: the theorems of the product \cite[Thm.2.26]{HN}, of the polygon \cite[Thm.3.1]{HN}, and of the residual polynomial \cite[Thm.3.7]{HN}. Let us only mention the following facts, which are an immediate consequence of these results.

\begin{proposition}\label{onesided}
Let $\ty$ be a type of order $r\ge 1$, and let $g(x)\in\zpx$ be a nonzero polynomial. Then, for any $1\le i\le r$:
\begin{enumerate}
\item[(i)] If $N_i^-(g)$ is reduced to a point or it is one-sided of slope $\lambda\ne\lambda_i$, then $R_i(g)(y)$ is a constant in $\ff{i}$.
\item[(ii)] If $N_i^-(g)$ is not reduced to a point and $g(x)$ is irreducible, then $N_i(g)=N_i^-(g)$ is one-sided and  $N_j(g)=N_j^-(g)$ is one-sided of slope $\lambda_j$, for all $1\le j<i$. Moreover, if $\beta\in\qpb$ is a root of $g(x)$, then the Theorem of the polygon shows that
\begin{equation}\label{thpolygon}
v(\phi_i(\beta))=\dfrac{|\lambda|+v_i(\phi_i)}{e_1\cdots e_{i-1}},
\end{equation}
where $\lambda$ is the slope of $N_i^-(g)$. 
\item[(iii)]  If $N_i(g)$ is one-sided of slope $\lambda_i$, then  $\deg(R_i(g))=\deg(g)/e_i\deg(\phi_i)$. 
If moreover $g(x)$ is irreducible, then $R_i(g)(y)$ is the power of an irreducible polynomial in $\ff{i}[y]$. 
\item[(iv)] If $i<r$, then $N_{i+1}^-(g)$ has length equal to $\operatorname{ord}_{\psi_i(y)}(R_i(g)(y))$.
\end{enumerate}
\end{proposition}
 
\begin{defin}\label{defs}
 Let $f(x)\in\zpx$ be a monic separable polynomial, and $\ty$ a type of order $r\ge 1$.

(1) \ We say that $\ty$ \emph{divides} $f(x)$, if $\psi_r(y)$ divides $R_r(f)(y)$ in $\ff{r}[y]$.

(2) \ We say that $\ty$ is \emph{$f$-complete} if $\operatorname{ord}_{\psi_r}(R_r(f))=1$. In this case, $\ty$ singles out an irreducible factor of $f(x)$ in $\zpx$. This factor is denoted $f_\ty(x)$; it has degree $e_rf_r\deg(\phi_r)$, and it is uniquely determined by the property  $R_r(f_\ty)(y)\sim\psi_r(y)$. 

(3) \ A \emph{representative} of $\ty$ is a monic polynomial $\phi_{r+1}(x)\in\zpx$, of mi\-nimum degree having the property that $\ty$ is $\phi_{r+1}$-complete. This polynomial is necessarily irreducible in $\zpx$. Note that, by the definition of a type, each $\phi_i(x)$ is a representative of the truncated type of order $i-1$
$$
\ty_{i-1}:=(\phi_1(x);\lambda_1,\phi_2(x);\cdots;\lambda_{i-2},\phi_{i-1}(x);\lambda_{i-1},\psi_{i-1}(y)).
$$
In \cite[Sect.2.3]{HN} we found an explicit and efficient procedure to compute representatives of arbitrary types. 

(4) \ We say that $\ty$ is \emph{optimal} if $\deg(\phi_1)<\cdots<\deg(\phi_r)$, or equivalently, $e_if_i>1$, for all $1\le i< r$.

We say that $\ty$ is \emph{strongly optimal} if $\ty$ is optimal and $e_rf_r>1$.  

We convene that all types of order zero are strongly optimal. 
\end{defin}

We finish this section with a proposition, extracted from \cite[Prop.3.5]{HN}, which plays an essential role in what follows, and an auxiliary lemma.

\begin{proposition}\label{vgt}
Let $F(x)\in\zpx$ be a monic irreducible polynomial, $\t\in\qpb$ a root of $F(x)$, and $\,\ty=(\phi_1(x);\lambda_1,\phi_2(x);\cdots;\lambda_{r-1},\phi_r(x);\lambda_r,\psi_r(y))$ a type of order $r$ dividing $F(x)$.  
Let $g(x)\in\zpx$ be a nonzero polynomial. Take a line of slope $\lambda_r$ far below $N_r(g)$, and let it shift upwards till it touches the polygon for the first time. Let $H$ be the ordinate at the origin of this line. Then,  
$v(g(\t))\ge H/e_1\cdots e_{r-1}$, and equality holds if and only if $\ty$ does not divide $g(x)$.
\end{proposition}

\begin{lem}\label{aux2}
Let $F(x)\in\zpx$ be a monic irreducible polynomial, and let $\ty=(\phi_1(x);\lambda_1,\phi_2(x);\cdots;\lambda_{r-1},\phi_r(x);\lambda_r,\psi_r(y))$ be a type of order $r\ge 1$ dividing $F(x)$. Let $\phi_{r+1}(x)$  a representative of $\ty$, and let $\lambda_{r+1}=-h_{r+1}/e_{r+1}$ be the slope of $N_{r+1}(F)$, where $h_{r+1},e_{r+1}$ are positive coprime integers.
Then, $$v(\phi_{r+1}(\t))=\dfrac{\deg(\phi_{r+1})}{\deg(\phi_r)}\,v(\phi_r(\t))+\dfrac{h_{r+1}}{e_1\cdots e_{r+1}}.
$$ 
\end{lem}

\begin{proof}
By \cite[Prop.2.7]{HN},  $v_{r+1}(\phi_r)=e_rv_r(\phi_r)+h_r$, and by \cite[Thm.2.11]{HN}, $v_{r+1}(\phi_{r+1})=e_rf_rv_{r+1}(\phi_r)$. These two formulas together show that
$$
\dfrac{v_{r+1}(\phi_{r+1})}{e_1\cdots e_r}=e_rf_r\,\dfrac{|\lambda_r|+v_r(\phi_r)}{e_1\cdots e_{r-1}}.
$$
Therefore, since $\deg(\phi_{r+1})/\deg(\phi_r)=e_rf_r$, the lemma is a consequence of the Theorem of the polygon, (\ref{thpolygon}):
$$
v(\phi_{r+1}(\t))=\dfrac{|\lambda_{r+1}|+v_{r+1}(\phi_{r+1})}{e_1\cdots e_r},\quad
v(\phi_r(\t))=\dfrac{|\lambda_r|+v_r(\phi_r)}{e_1\cdots e_{r-1}}.
$$
\end{proof}

\subsection{Okutsu frames and complete types}

\begin{thm}\label{niam}
Let $[F_1,\dots,F_r]$ be an Okutsu frame of $F(x)$, and let  $K_{r+1}=L$,  $F_{r+1}(x)=F(x)$. Then, there exist negative rational numbers $\lambda_1,\dots,\lambda_r$, a finite field $\ff{r}$ and a monic irreducible polynomial $\psi_r(y)\in\ff{r}[y]$, such that 
$$
(F_1(x);\lambda_1,F_2(x);\cdots;\lambda_{r-1},F_r(x);\lambda_r,\psi_r(y)),
$$ 
is an $F$-complete strongly optimal type of order $r$. More precisely, $F_1$ is irreducible modulo $\m$, the finite field $\ff{1}:=\zpx/(\m[x],F_1(x))$ is the residue field of $K_1$, and for all $1\le i\le r$:
\begin{enumerate}
\item $N_i(F)$ and $N_i(F_{i+1})$ are one-sided of slope $\lambda_i$.
\item If we write $\lambda_i=-h_i/e_i$, with $h_i,e_i$ positive coprime integers, then   $e_i=e(K_{i+1}/K)/e(K_i/K)$.
\item $R_i(F)(y)\sim \psi_i(y)^{a_i}$, for some monic irreducible polynomial $\psi_i(y)\in\ff{i}[y]$, and $R_i(F_{i+1})(y)\sim \psi_i(y)$. 
\item $f_i:=\deg \psi_i(y)=f(K_{i+1}/K)/f(K_i/K)$.
\item $\ff{i+1}:=\ff{i}[y]/(\psi_i(y))$ has degree $[\ff{i+1}\colon \ff{}]=f(K_{i+1}/K)$. 
\end{enumerate}
\end{thm}
   
\begin{proof}
By (ii) of Corollary \ref{phi1}, $F(x)\equiv F_1(x)^{\deg(F)/m_1}\md{\m}$ and $F_1(x)$ is irreducible modulo $\m$; in particular, $\ff1$ is the residue field of $K_1$. Let us prove simultaneously all statements (1)-(5) by induction on $i$. Actually, we argue with an arbitrary $1\le i\le r$; if $i=1$ we do not make any assumption and if $i>1$ we assume that the conditions (1)-(5) are true for the indices $1,\dots,i-1$, for all Okutsu frames of all polynomials of depth less than or equal to $r$.

If $i=1$, the Newton polygon $N_1^-(F)$ has length $\deg(F)/m_1$  \cite[Def.1.8]{HN}. If $i>1$, the induction hypothesis shows that $F_i(x)$ is a representative of the type
$(F_1(x);\cdots,F_{i-1}(x);\lambda_{i-1},\psi_{i-1}(y))$, and $\psi_{i-1}(y)$ divides $R_{i-1}(F)(y)$; hence,  
$N_i^-(F)$ has positive length equal to $\operatorname{ord}_{\psi_{i-1}(y)}(R_{i-1}(F)(y))$, by (iv) of Proposition \ref{onesided}. Therefore, $N_i^-(F)$ is never reduced to a point. By (ii) of Proposition \ref{onesided}, $N_i^-(F)=N_i(F)$ is one-sided and it has a negative slope that we denote by $\lambda_i$, and we write as $\lambda_i=-h_i/e_i$, with $h_i,e_i$ positive coprime integers. Also, by (iii) of Proposition  \ref{onesided},  $R_i(F)(y)$ is, up to a multiplicative constant, the power of some monic irreducible polynomial $\psi_i(y)\in\ff{i}[y]$, whose degree we denote by $f_i$. 

\begin{center}
\setlength{\unitlength}{5.mm}
\begin{picture}(22,5)
\put(-.2,.85){$\bullet$}
\put(-1,0){\line(1,0){4}}\put(0,-1){\line(0,1){6}}
\put(1.5,-.4){\line(-1,1){2}}
\put(-.6,-.6){\begin{scriptsize}$0$\end{scriptsize}}
\put(.7,.5){\begin{scriptsize}$\lambda_i$\end{scriptsize}}
\put(.8,-1.6){\begin{scriptsize}$N^-_i(g)$\end{scriptsize}}
\put(-1.6,.9){\begin{scriptsize}$v_i(g)$\end{scriptsize}}

\put(6.8,3.75){$\bullet$}\put(11.85,.8){$\bullet$}
\put(6,0){\line(1,0){7}}\put(7,-1){\line(0,1){6}}
\put(7,4){\line(5,-3){5}}\put(7,3.96){\line(5,-3){5}}
\put(6.1,5.2){\line(3,-4){4.5}}
\put(6.4,-.6){\begin{scriptsize}$0$\end{scriptsize}}
\put(6.2,3.8){\begin{scriptsize}$H$\end{scriptsize}}
\put(11.8,-.65){\begin{scriptsize}$m$\end{scriptsize}}
\put(9.1,2.8){\begin{scriptsize}$\lambda$\end{scriptsize}}
\put(10.4,-.6){\begin{scriptsize}$\lambda_i$\end{scriptsize}}
\put(7.8,-1.6){\begin{scriptsize}$N^-_i(g)=N_i(g)$\end{scriptsize}}
\multiput(6.9,1)(.25,0){20}{\hbox to 2pt{\hrulefill }}
\put(5.4,.9){\begin{scriptsize}$v_i(g)$\end{scriptsize}}
\multiput(12,-.1)(0,.25){5}{\vrule height2pt}

\put(16.8,3.75){$\bullet$}\put(21.85,.8){$\bullet$}
\put(16,0){\line(1,0){7}}\put(17,-1){\line(0,1){6}}
\put(17,4){\line(5,-3){5}}\put(17,3.96){\line(5,-3){5}}
\put(23.5,.4){\line(-5,2){7.5}}
\put(16.9,3){\line(1,0){.3}}
\put(16.4,-.6){\begin{scriptsize}$0$\end{scriptsize}}
\put(16.3,2.6){\begin{scriptsize}$H$\end{scriptsize}}
\put(21.8,-.65){\begin{scriptsize}$m$\end{scriptsize}}
\put(19.1,2.8){\begin{scriptsize}$\lambda$\end{scriptsize}}
\put(22.8,.7){\begin{scriptsize}$\lambda_i$\end{scriptsize}}
\put(17.8,-1.6){\begin{scriptsize}$N^-_i(g)=N_i(g)$\end{scriptsize}}
\multiput(16.9,1)(.25,0){20}{\hbox to 2pt{\hrulefill }}
\put(15.4,.9){\begin{scriptsize}$v_i(g)$\end{scriptsize}}
\multiput(22,-.1)(0,.25){5}{\vrule height2pt}
\end{picture}
\end{center}\bigskip
\begin{center}
Figure 1 
\end{center}

We want to show that $N_i(F_{i+1})$ is also one-sided of the same slope $\lambda_i$, and $R_i(F_{i+1})\sim \psi_i(y)$. To this end we apply Proposition \ref{vgt} to the polynomial $F_{i+1}(x)$ and the type $\ty_i:=(F_1(x);\cdots,F_i(x);\lambda_i,\psi_i(y))$, in order to estimate $v(F_{i+1}(\t))$ and compare this estimation with the inequality 
\begin{equation}\label{greater}
v(F_{i+1}(\t))>\dfrac{m_{i+1}}{m_i}\,v(F_i(\t))=\dfrac{m_{i+1}}{m_i}\,\dfrac{|\lambda_i|+v_i(F_i)}{e_1\cdots e_{i-1}},
\end{equation}
given by Lemma \ref{comparison} and the Theorem of the polygon, (\ref{thpolygon}).
The possible shapes of $N_i^-(F_{i+1})$ and the first point of contact with a line of slope $\lambda_i$ shifting upwards from below are displayed in Figure 1, where we denote $g(x):=F_{i+1}(x)$ and $m:=m_{i+1}/m_i$.

Let $F_{i+1}(x)=\sum_{0\le j< m_{i+1}/m_i}a_j(x)F_i(x)^j+F_i(x)^{m_{i+1}/m_i}$ be the $F_i$-adic deve\-lopment of $F_{i+1}$. By \cite[Lem.2.17,(1)]{HN}, we have
$$v_i(F_{i+1})=\min_{0\le j\le m_{i+1}/m_i}\{v_i(a_jF_i^j)\}.$$ If $v_i(a_0)=v_i(F_{i+1})$, then $N_i^-(F_{i+1})$ is reduced to the point $(0,v_i(F_{i+1}))$ and Proposition \ref{vgt} shows that $$v(F_{i+1}(\t))=  
v_i(F_{i+1})/e_1\cdots e_{i-1}\le v_i(F_i^{m_{i+1}/m_i})/e_1\cdots e_{i-1},$$ in contradiction with (\ref{greater}). Hence, $v_i(a_0)>v_i(F_{i+1})$; this implies that $N_i^-(F_{i+1})$ is not reduced to a point, and by (ii) of Proposition \ref{onesided}, $N_i(F_{i+1})=N^-_i(F_{i+1})$ is one-sided and it has a negative slope, that we denote by $\lambda$. In particular, $v_i(F_{i+1})=(m_{i+1}/m_i)v_i(F_i)$.
     
The ordinate at the origin of the line of slope $\lambda_i$ that first touches the polygon from below is equal to $$H=v_i(F_{i+1})+\dfrac{m_{i+1}}{m_i}\min\{|\lambda_i|,|\lambda|\}=\dfrac{m_{i+1}}{m_i}\left(v_i(F_i)+\min\{|\lambda_i|,|\lambda|\}\right).$$ If $\lambda\ne\lambda_i$, item (i) of Proposition \ref{onesided} shows that $R_i(F_{i+1})$ is a constant. Hence, if either  $\lambda\ne\lambda_i$ or $
\psi_i(y)$ does not divide $R_i(F_{i+1})$, Proposition \ref{onesided} would imply $v(F_{i+1}(\t))=H/e_1\cdots e_{i-1}$, in contradiction with (\ref{greater}). Therefore, $\lambda=\lambda_i$ and $R_i(F_{i+1})(y)\sim \psi_i(y)^a$, for some positive exponent $a$.

Denote $e'_j:=e(K_{j+1}/K)/e(K_j/K)$, $f'_j:=f(K_{j+1}/K)/f(K_j/K)$, for all $1\le j\le i$. By the induction hypothesis, $e'_j=e_j$, $f'_j=f_j$, for all $1\le j<i$. Note that $e(K_1/K)=1$, $f(K_1/K)=m_1$. By (iii) of Proposition \ref{onesided} applied to $F_{i+1}(x)$, we have
$$
m_ie'_if'_i=e(K_{i+1}/K)f(K_{i+1}/K)=m_{i+1}=m_ie_if_ia.
$$ 
By \cite[Cors.1.16,1.20]{HN} we have 
$$\as{1.4}
\begin{array}{l}
e_1\cdots e_i\,|\,e(K_{i+1}/K)=e_1\cdots e_{i-1}e'_i, \\ 
 {}[ \ff{i+1}\colon \ff{} ]=m_1f_1\cdots f_i\,|\,f(K_{i+1}/K)=m_1f_1\cdots f_{i-1}f'_i.
\end{array}
$$ 
In particular, $e_i\,|\,e'_i$, $f_i\,|\,f'_i$. Thus, if we prove that $a=1$, then all statements (1)-(5) will be proven. Let  $\phi_{i+1}(x)\in\zpx$ be a representative of the above mentioned type $\ty_i$. This monic irreducible polynomial has degree $e_if_im_i$ \cite[Thm.2.11]{HN}, and by Lemma \ref{aux2}, $v(\phi_{i+1}(\t))>e_if_iv(F_i(\t))$. Now, $a>1$ implies $\deg(\phi_{i+1}(x))<m_{i+1}$, and this leads to $v(\phi_{i+1}(\t))\le v(F_i(\t)^{e_if_i})$, by Theorem \ref{allm}. Therefore, $a=1$ and the theorem is proven.
\end{proof}

By \cite[Cor.3.2,(1)]{HN},  Theorem \ref{niam} has the following immediate consequence.

\begin{cor}\label{vFi}
The invariants $v(F_1(\t)),\dots,v(F_r(\t))$ can be computed directly in terms of any Okutsu frame as:
$$v(F_i(\t))=\sum_{j=1}^i(e_jf_j\cdots e_{i-1}f_{i-1})\dfrac{h_j}{e_1\cdots e_j}.$$\qed 
\end{cor}

\begin{cor}\label{newinvariants}
The following invariants depend only on $F(x)$:
\begin{enumerate}
\item[(i)] The slopes $\lambda_1,\dots,\lambda_r$, respectively of $N_1(F),\dots,N_r(F)$.
\item[(ii)] The discrete valuations $v_1,\dots,v_{r+1}$ on $K(x)$.
\end{enumerate}
The slopes $\lambda_1,\dots,\lambda_{r-1}$, and the discrete valuations $v_1,\dots,v_r$ on $K(x)$, can be computed from any Okutsu frame of $F(x)$.
\end{cor}

\begin{proof}
Suppose that $[F_1',\dots,F'_r]$ is another Okutsu frame of $F(x)$, leading to the family of slopes $\lambda'_1,\dots,\lambda'_r$, and valuations $v_1,v'_2,\dots,v'_{r+1}$ on $K(x)$. By Theorem \ref{niam}, both Okutsu frames determine $F$-complete types of order $r$; let $h_i,e_i,f_i$, and  $h'_i,e'_i,f'_i$ be the respective invariants of these types, for all $1\le i\le r$. By Theorem \ref{niam} and Corollary \ref{invariants}, $$e_i=e(K_{i+1}/K)/e(K_i/K)=e'_i, \quad f_i=f(K_{i+1}/K)/f(K_i/K)=f'_i,$$for all $1\le i\le r$.
Let us prove that $\lambda_i=\lambda'_i$ for all $1\le i\le r$, by induction on $i$. Since $F_1(x)$ and $F'_1(x)$ are both divisor polynomials of degree $m_1$ of $F(x)$ (Theorem \ref{PDP}), we have
$\lambda_1=-v(F_1(\t))=-v(F'_1(\t))=\lambda'_1$. Suppose $i>1$ and 
$\lambda_j=\lambda'_j$ (hence $h_j=h'_j$), for all $1\le j<i$. Since $F_i(x)$ and $F'_i(x)$ are both divisor polynomials of degree $m_i$ of $F(x)$, we have $v(F_i(\t))=v(F'_i(\t))$, and by the Theorem of the polygon (cf. (\ref{thpolygon})):
 \begin{equation}\label{compare}
\dfrac{|\lambda_i|+v_i(F_i)}{e_1\cdots e_{i-1}}=v(F_i(\t))=v(F'_i(\t))=\dfrac{|\lambda'_i|+v'_i(F'_i)}{e_1\cdots e_{i-1}}.
\end{equation}
In \cite[Prop.2.15]{HN}, we found a closed formula for $v_i(F_i)$ that depends only on 
 $h_1,\dots,h_{i-1}$, $e_1,\dots,e_{i-1}$ and $f_1,\dots,f_{i-1}$; hence, $v_i(F_i)=v'_i(F'_i)$, and (\ref{compare}) shows that $\lambda_i=\lambda'_i$.

Let us prove now that $v_i=v'_i$ for all $1\le i\le r+1$, by induction on $i$. The valuation $v_1$ being canonical, we need only to prove that $v_{i+1}=v'_{i+1}$, under the assumption $v_i=v'_i$, for some $1\le i\le r$.
We claim that $N_i(F'_i)$ is one-sided with negative slope $-h$, for some positive integer $h\ge |\lambda_i|$, which may be taken to be $h=\infty$ if $F_i=F'_i$ \cite[Lem.2.17,(3)]{HN}. In fact, let $A(x)=F_i(x)-F_i'(x)$, with $v_i(A)=v_i(F_i)+h$; the shape of $N_i(F'_i)$ is displayed in the following picture:
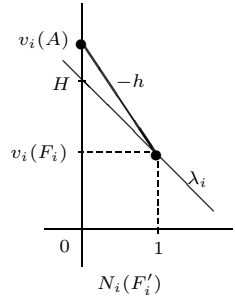
\begin{center}
\setlength{\unitlength}{5.mm}
\begin{picture}(5,6)
\put(-.2,4.75){$\bullet$}\put(1.8,1.8){$\bullet$}
\put(-1,0){\line(1,0){5}}\put(0,-1){\line(0,1){7}}
\put(0,5){\line(2,-3){2}}\put(0,4.96){\line(2,-3){2}}
\put(3.5,.5){\line(-1,1){4}}\put(-.1,3.95){\line(1,0){.2}}
\put(-1.7,4.9){\begin{scriptsize}$v_i(A)$\end{scriptsize}}
\put(2.8,1.2){\begin{scriptsize}$\lambda_i$\end{scriptsize}}
\put(-.8,3.7){\begin{scriptsize}$H$\end{scriptsize}}
\put(-.6,-.6){\begin{scriptsize}$0$\end{scriptsize}}
\put(1.9,-.65){\begin{scriptsize}$1$\end{scriptsize}}
\put(.9,3.8){\begin{scriptsize}$-h$\end{scriptsize}}
\put(.4,-1.6){\begin{scriptsize}$N_i(F'_i)$\end{scriptsize}}
\multiput(-.1,2.05)(.25,0){9}{\hbox to 2pt{\hrulefill }}
\put(-1.9,1.9){\begin{scriptsize}$v_i(F_i)$\end{scriptsize}}
\multiput(2,-.1)(0,.25){9}{\vrule height2pt}
\end{picture}
\end{center}\bigskip\bigskip
and $v(F'_i(\t))\ge H/e_1\cdots e_{i-1}$, by Proposition \ref{vgt}. This implies  $h\ge|\lambda_i|$, because $v(F'_i(\t))=(|\lambda_i|+v_i(F_i))/e_1\cdots e_{i-1}$, by (\ref{compare}).

Both discrete valuations $v_{i+1},\,v'_{i+1}$ coincide with $e_1\cdots e_i\, v_1$, when res\-tricted to $K$ \cite[Prop.2.6]{HN}. Hence, in order to show that they are equal it is suficient to show that they coincide on all  irreducible polynomials with coefficients in $\oo$. Let $g(x)\in\zpx$ be an irreducible polynomial. Let us recall the definition of $v_{i+1}$ \cite[Def.2.5]{HN}. If we consider a line of slope $\lambda_i$ far below $N_i(g)$ and we let it shift upwards till it touches the polygon for the first time, then $v_{i+1}(g(x))=e_i H$, where $H$ is the ordinate at the origin of this line. Let
$g(x)=\sum_{0\le j\le m}a_j(x)F_i(x)^j$ be the $F_i$-adic development of $g(x)$; we recall that $v_i(g)=\min_{0\le j\le m}\{v_i(a_j(x)F_i(x)^j)\}$ \cite[Lem.2.17,(1)]{HN}. Note that the $0$-th coefficient of the $F'_i$-adic development, $g(x)=\sum_{0\le j\le m}a'_j(x)F'_i(x)^j$, is $a'_0=a_0+a_1A+\cdots a_m A^m$. We distinguish three cases according to the shape of $N_i^-(g)$, and its first point of contact with the line of slope $\lambda_i$ under it. These possibilities are reflected in Figure 1, that was displayed along the proof of Theorem \ref{niam}.

If $v_i(a_0)=v_i(g)$, then $N_i^-(g)$ is reduced to the point $(0,v_i(g))$, and $v_{i+1}(g)=e_iv_i(g)$. Since $v_i(A)>v_i(F_i)$, 
we have $v_i(a_0)\le v_i(a_kF_i^k)<v_i(a_kA^k)$, for all $k>0$. Thus, $v_i(a'_0)=v_i(a_0)=v_i(g)$, so that $(N'_i)^-(g)$ is also reduced to the point $(0,v_i(g))$, and $v'_{i+1}(g)=e_iv_i(g)$ too.
  
If $v_i(a_0)>v_i(g)$, then $N_i^-(g)$ is not reduced to a point, and $N^-_i(g)=N_i(g)$ is one-sided of negative slope $\lambda$, by (ii) of Proposition \ref{onesided}. In the special case $g(x)=F_i(x)$ we take $\lambda=\infty$ \cite[Lem.2.17,(3)]{HN}. As Figure 1 shows,  $v_{i+1}(g)=e_i H$, where  
$H=v_i(g)+m\min\{|\lambda_i|,|\lambda|\}$. Now, 
\begin{equation}\label{a0}
 v_i(a_kA^k)=v_i(a_kF_i^k)+kh\\\ge v_i(g)+(m-k)|\lambda|+kh,
\end{equation}
for all $k\ge0$.
Suppose $|\lambda|<|\lambda_i|$. Then, $h>|\lambda|$, and $v_i(a_kA^k)>v_i(g)+m|\lambda|=v_i(a_0)$, for all $k>0$.
Hence, $v_i(a'_0)=v_i(a_0)>v_i(g)$. This implies $N_i^-(g)=(N'_i)^-(g)$, because both polygons are one-sided with the same end points: $(0,v_i(a_0))$ and $(m,v_i(g))$. In particular, $v_{i+1}(g)=v'_{i+1}(g)=e_iv_i(a_0)$.
    
Suppose now $|\lambda|\ge|\lambda_i|$. By (\ref{a0}), we have now $v_i(a_kA^k)\ge v_i(g)+m|\lambda_i|$, for all $k\ge0$. Hence, $v_i(a'_0)\ge v_i(g)+m|\lambda_i|>v_i(g)$, and $(N'_i)^-(g)$ is one-sided of negative slope $\lambda'$, which is not necessarily equal to $\lambda$. Clearly,  $|\lambda'|=(v_i(a'_0)-v_i(g))/m\ge|\lambda_i|$, and  $v_{i+1}(g)=v'_{i+1}(g)=e_i(v_i(g)+m|\lambda_i|)$.

The last statement of the corollary is a consequence of Theorem \ref{niam}: $\lambda_i$ is the slope of $N_i(F_{i+1})$, and $v_{i+1}$ is determined by $F_i(x)$, $v_i$ and $\lambda_i$.
\end{proof}

Our next aim is to prove the converse of Theorem \ref{niam}: an $F$-complete strongly optimal type is an Okutsu frame, plus the data $\lambda_r$, $\psi_r(y)$. To this end we recall what kind of optimal behaviour do have the optimal types. The following result is extracted from \cite[Thm.3.1]{GMN}, where this optimality was analyzed in a more general situation.

\begin{thm}{\cite[Thm.3.1]{GMN}}\label{optimal}
Let $F(x)\in\zpx$ be a monic irreducible polynomial, and $\ty=(\phi_1(x);\lambda_1,\phi_2(x);\cdots;\lambda_{i-1},\phi_i(x);\lambda_i,\psi_i(y))
$ a type of order $i\ge 1$ dividing $F(x)$. Let $\phi'_i(x)\in\zpx$ be another representative of the truncated type 
$$\ty_{i-1}=(\phi_1(x);\lambda_1,\phi_2(x);\cdots;\lambda_{i-2},\phi_{i-1}(x);\lambda_{i-1},\psi_{i-1}(y)).$$
Let $\lambda'_i$ be the slope of the one-sided Newton polygon of $i$-th order of $F(x)$, $N_i'(F)$, taken with respect to the pair $(\phi'_i(x),v_i)$. If $e_if_i>1$, then $|\lambda'_i|
\le |\lambda_i|$.
\end{thm}

\begin{thm}\label{main}
Let $F(x)\in\zpx$ be a monic irreducible polynomial, and let $\ty=(\phi_1(x);\lambda_1,\phi_2(x);\cdots;\lambda_{r-1},\phi_r(x);\lambda_r,\psi_r(y))$ be an $F$-complete strongly optimal type of order $r\ge 1$.
Then, $[\phi_1,\dots,\phi_r]$ is an Okutsu frame of $F(x)$.
\end{thm}

\begin{proof}
Let $m'_i=\deg(\phi_i)$, for all $1\le i\le r$, and $m'_{r+1}:=\deg(F)$. We know that $m'_{i+1}=m'_ie_if_i$ for all $1\le i\le r$.
Let $s=\dep(F)$ and consider the basic Okutsu invariants:
$$
m_1<\cdots<m_s<m_{s+1}=\deg(F),\quad \mu_1<\cdots<\mu_s<\mu_{s+1}=\infty.
$$ 

Since $\phi_1(x)$ is irreducible modulo $\m$, $F(x)\equiv \phi_1(x)^{\deg(F)/m'_1}\md{\m}$,
and $\deg(F)>m'_1$, Corollary \ref{phi1} shows that $s\ge1$ and $m'_1=m_1$. Let $F_1(x)\in\zpx$ be any choice for the first polynomial of an Okutsu frame. The Newton polygons $N_{\phi_1}(F)$, $N_{F_1}(F)$ are both one-sided, of respective slope $-v(\phi_1(\t))$, $-v(F_1(\t))$. Since $F_1(x)$ is a divisor polynomial of degree $m_1$ of $F(x)$ (Theorem \ref{PDP}), we have $v(\phi_1(\t))\le v(F_1(\t))$. Since $m'_2>m'_1$, Theorem \ref{optimal} shows that the opposite inequality holds, so that $v(\phi_1(\t))= v(F_1(\t))$. Now, by Theorem \ref{PDP}, $\phi_1(x)$ can be taken too as the first polynomial of an Okutsu frame of $F(x)$.

Suppose that for some $1\le i\le r$, the polynomials $\phi_1,\dots,\phi_i$ can be taken as the first $i$ polynomials of an Okutsu frame of $F(x)$. Let us show that in this case, $i=r$ if and only if $i=s$. In fact, suppose $i=r<s$, and let $F_{r+1}$ be any $(r+1)$-th polynomial in an Okutsu frame. Since $\ty$ is $F$-complete, $N_r(F)$ is one-sided of slope $\lambda_r$ and $R_r(F)\sim\psi_r(y)$; on the other hand, Theorem \ref{niam} shows that $N_r(F_{r+1})$ is one-sided of the same slope $\lambda_r$, and  $R_r(F_{r+1})\sim\psi_r(y)$ too. By (iii) of Proposition \ref{onesided},  $m_{r+1}=m_re_rf_r=m'_{r+1}=\deg(F)$, and this contradicts the assumption $s>r$. Suppose now $i=s<r$. By Lemma \ref{aux2}, $v(\phi_{s+1}(\t))>(m'_{s+1}/m'_s)v(\phi_s(\t))$, and we get again a contradiction, since $m'_{s+1}<m'_{r+1}=\deg(F)$ implies $v(\phi_{s+1}(\t))\le v(\phi_s(\t)^{m'_{s+1}/m'_s})$, by Theorem \ref{allm}.

Finally, suppose that for some $1\le i<r$, the polynomials $\phi_1,\dots,\phi_i$ can be taken as the first $i$ polynomials of an Okutsu frame of $F(x)$. We have just seen that $i<s$ too; thus, there exists $F_{i+1}(x)$ such that  $\phi_1,\dots,\phi_i,\,F_{i+1}$ are the first $i+1$ polynomials of an Okutsu frame of $F(x)$. By Theorem \ref{niam},  $N_i(F_{i+1})$ is one-sided of the same slope $\lambda_i$ as $N_i(F)$, and $R_i(F_{i+1})\sim \psi_i(y)$, the monic irreducible factor of $R_i(F)(y)$; hence,  $m_{i+1}=m_ie_if_i=m'_ie_if_i=m'_{i+1}$. In particular, $F_{i+1}$ is another representative of the truncated type $$\ty_i=(\phi_1(x);\lambda_1,\phi_2(x);\cdots;\lambda_{i-1},\phi_i(x);\lambda_i,\psi_i(y)),$$ and $v_{i+1}(F_{i+1})=v_{i+1}(\phi_{i+1})$, since this value depends only on $\lambda_1,\dots,\lambda_i$ and $f_1,\dots,f_i$ \cite[Prop.2.15]{HN}. By Theorem \ref{optimal}, $|\lambda'_{i+1}|\le |\lambda_{i+1}|$, where $\lambda'_{i+1}$ is the slope of $N_{F_{i+1},v_{i+1}}(F)$, and by the Theorem of the polygon, (\ref{thpolygon}), we have $v(F_{i+1}(\t))\le v(\phi_{i+1}(\t))$. 
Now, since $F_{i+1}$ is a divisor polynomial of degree $m_{i+1}$ of $F(x)$ (Theorem \ref{PDP}), we have $v(F_{i+1}(\t))=v(\phi_{i+1}(\t))$. Hence, $\phi_{i+1}(x)$ is a divisor polynomial of degree $m_{i+1}$ of $F(x)$, and Theorem \ref{PDP} shows that $\phi_1,\dots,\phi_i,\phi_{i+1}$ are the first $i+1$ polynomials of an Okutsu frame of $F(x)$.
\end{proof}

\section{Montes approximations to an irreducible polynomial}\label{montes}
Although Theo\-rems \ref{niam} and \ref{main} seem to indicate that the notion of Okutsu frame is equivalent to that of $F$-complete strongly optimal type, this is not quite exact. An $F$-complete strongly optimal type is an Okutsu frame together with extra information on $F(x)$ given by the data $\lambda_r$, $\psi_r(y)$. 

In this section we introduce the notion of \emph{Montes approximation} to $F(x)$; this is a monic irreducible polynomial in $\zpx$ of the same degree and sufficiently close to $F(x)$. A better way to summarize the content of Theorems \ref{niam} and \ref{main} is to think that an $F$-complete strongly optimal type is equiva\-lent to an Okutsu frame, together with a Montes approximation to $F(x)$. After extending this point of view to the optimal case (Theorem \ref{quasimain}), Montes algorithm can be reinterpreted as a fast method to compute an Okutsu frame and a Montes approximation to each of the irreducible factors of a monic separable polynomial in $\zpx$ (Section \ref{reinterpr}). Finally, Newton polygons can be used too to obtain approximations with arbitrary prescribed precision, leading in this way to a factorization algorithm (Section \ref{factorization}).

\subsection{Montes approximations}
\begin{lemma-definition}\label{Mapp}
Let $F(x)\in\zpx$ be a monic irreducible polynomial of degree $n$, $\t\in\qpb$ a root of $F(x)$ and $r=\dep(F)$.  
For any $\phi(x)\in\zpx$, the following three conditions are equivalent: 
\begin{enumerate}
\item[(i)] $\phi(x)$ is a monic irreducible polynomial of degree $n$, and it has a root $\alpha\in \qpb$ satisfying $v(\t-\alpha)>\mu_r$.
\item[(ii)] $\phi(x)$ is a monic irreducible polynomial of degree $n$, and $$v(\phi(\t))>(n/m_r)v(F_r(\t)),$$ where $F_r(x)$ is the $r$-th polynomial of an Okutsu frame of $F(x)$.
\item[(iii)]  $\phi(x)$ is a representative of some $F$-complete strongly optimal type of order $r$.
\end{enumerate}

If any of these conditions is  satisfied, we say that $\phi(x)$ is a \emph{Montes approximation} to $F(x)$. 
\end{lemma-definition}

\begin{proof}
Lemma \ref{comparison} shows that (i) and (ii) are equivalent. Let us show that (ii) and (iii) are equivalent too.
Let $[F_1,\dots,F_r]$ be an Okutsu frame of $F(x)$, and suppose that $\phi(x)$ satisfies (ii). By Theorem \ref{niam}, there are negative rational numbers $\lambda_1,\dots,\lambda_r$ and a monic irreducible polynomial $\psi_r(y)$ over some finite field, such that the type 
$(F_1(x);\lambda_1,F_2(x);\cdots;\lambda_{r-1},F_r(x);\lambda_r,\psi_r(y))$ is $F$-complete and strongly optimal. Arguing as in the proof of Theorem \ref{niam}, we deduce from Proposition \ref{vgt} that $\phi(x)$ is a representative of this type.

Conversely, suppose $\phi(x)$ is a representative of an $F$-complete strongly optimal type $\,\ty=(\phi_1(x);\lambda_1,\phi_2(x);\cdots;\lambda_{r-1},\phi_r(x);\lambda_r,\psi_r(y))$. Then, $\phi(x)$ is a monic irreducible polynomial of degree $n$  \cite[Prop.3.12]{HN}. By Theo\-rem \ref{main}, $[\phi_1,\dots,\phi_r]$ is an Okutsu frame of $F(x)$, and Lemma \ref{aux2} shows that  $v(\phi(\t))>(n/m_r)v(\phi_r(\t))$.
\end{proof}

This concept leads to a reinterpretation of the arithmetic information contained in an $f$-complete optimal type, where $f(x)\in\zpx$ is a monic separable polynomial.

\begin{thm}\label{quasimain}
Let $f(x)\in\zpx$ be a monic separable polynomial, and let $\ty=(\phi_1(x);\lambda_1,\phi_2(x);\cdots;\lambda_{r-1},\phi_r(x);\lambda_r,\psi_r(y))
$ be an $f$-complete optimal type of order $r\ge 1$. Let $\phi_{r+1}(x)\in\zpx$ be a representative of $\ty$, and $f_\ty(x)\in\zpx$ the irreducible factor of $f(x)$ that corresponds to $\ty$. 

If $e_rf_r>1$, then $[\phi_1,\dots,\phi_r]$ is an Okutsu frame of $f_\ty(x)$ and $\phi_{r+1}(x)$ is a Montes approximation to $f_\ty(x)$.

If $e_rf_r=1$, then $[\phi_1,\dots,\phi_{r-1}]$ is an Okutsu frame of $f_\ty(x)$ and $\phi_r(x)$, $\phi_{r+1}(x)$ are both Montes approximations to $f_\ty(x)$. 
\end{thm}

\begin{proof}
If $e_rf_r>1$, the type $\ty$ is $f_\ty$-complete and strongly optimal, and the statement is a consequence of Theorem \ref{main} and Lemma-Definition \ref{Mapp}. 

If $e_rf_r=1$, let us show first that the truncated type $$\ty_{r-1}=(\phi_1(x);\lambda_1,\phi_2(x);\cdots;\lambda_{r-2},\phi_{r-1}(x);\lambda_{r-1},\psi_{r-1}(y))$$ is strongly optimal and $f_\ty$-complete. 
Since $\deg(\phi_r)=e_{r-1}f_{r-1}\deg(\phi_{r-1})$, we have 
$e_{r-1}f_{r-1}>1$, and $\ty_{r-1}$ is strongly optimal. Since $\ty$ is $f$-complete, we have $R_r(f_\ty)(y)\sim\psi_r(y)$. By (i) and (ii) of Proposition \ref{onesided}, $N_r(f_\ty)=N_r^-(f_\ty)$ is one-sided of slope $\lambda_r$; the length of this polygon is 
$\deg(f_\ty)/\deg(\phi_r)=e_rf_r=1$. By (iv) and (iii) of Proposition \ref{onesided}, $\operatorname{ord}_{\psi_{r-1}}(R_{r-1}(f_\ty))=1$, and  $R_{r-1}(f_\ty)(y)\sim\psi_{r-1}(y)$, so that $\ty_{r-1}$ is $f_\ty$-complete.

Therefore, $[\phi_1,\dots,\phi_{r-1}]$ is an Okutsu frame of $f_\ty(x)$ by Theorem \ref{main}. Since $\phi_r(x)$ is a representative of $\ty_{r-1}$, it is a Montes approximation to $f_\ty(x)$ by Lemma-Definition \ref{Mapp}. Finally, $\phi_{r+1}(x)$ is also a Montes approximation to $f_\ty(x)$, because $v(\phi_{r+1}(\t))>v(\phi_r(\t))$, by Lemma \ref{aux2}. 
\end{proof}

The following lemma is an immediate consequence of Lemma \ref{recurrence}.
\begin{lem}\label{sameframes}
If $\phi(x)$ is a Montes approximation to $F(x)$, then any Okutsu frame of $F(x)$ is an Okutsu frame of $\phi(x)$, and vice versa. In particular, the relation ``to be a Montes approximation to" is an equivalence relation on the set of all monic irreducible polynomials in $\zpx$.\qed
\end{lem}

This relation is strictly stronger than the equivalence relation ``to have the same Okutsu frames". For instance, consider two representatives $\phi(x)$, $\phi'(x)$, of two optimal types $\ty,\ty'$ of order $r$ that differ only on the last data: $(\lambda_r,\psi_r(y))\ne(\lambda'_r,\psi'_r(y))$, but they satisfy  $e_rf_r=e'_rf'_r$; then, $\phi(x)$, $\phi'(x)$ are monic irreducible polynomials of the same degree, having the same Okutsu frames, but they are not one a Montes approximation to the other. 

Thus, besides sharing with $F(x)$ all Okutsu invariants, a Montes approxi\-mation to $F(x)$ is close to $F(x)$ in some stronger sense, as the next result shows. 

\begin{proposition}\label{sameextension}
Let $F(x)\in\zpx$ be a monic irreducible polynomial of degree $n$, $\t\in\qpb$ a root of $F(x)$ and $L=K(\t)$. Let $\phi(x)\in\zpx$ be a Montes approximation to $F(x)$, $\alpha\in\qpb$ a root of $\phi(x)$ such that $v(\t-\alpha)>\mu_r$, and $N=K(\alpha)$. Then,
\begin{enumerate}
\item[(i)] $e(N/K)=e(L/K)$, $f(N/K)=f(L/K)$. 
\item[(ii)] The maximal tamely ramified subextension of $N/K$ is contained in $L/K$. In particular, if $L/K$ is tamely ramified then $L=N$.
\end{enumerate}
\end{proposition}

\begin{proof}By Lemma \ref{taylor}, $e(L/K)=e(N/K)$; this implies $f(L/K)=f(N/K)$, because $[N\colon K]=[L\colon K]$.
Item (ii) is a consequence of Proposition \ref{tame}.
\end{proof}

We end this section with a measure of the precision of these Montes approximations.
Let
$\ty=(\phi_1(x);\lambda_1,\phi_2(x);\cdots;\lambda_{r-1},\phi_r(x);\lambda_r,\psi_r(y))$
be an $F$-complete optimal type.  
Let $h_1,\dots,h_r$, $e_1,\dots,e_r$, $f_0,f_1,\dots,f_r$,  be the usual invariants of the type. Theorem \ref{niam} shows that $e(L/K)=e_1\cdots e_r$, $f(L/K)=f_0f_1\cdots f_r$. Consider the rational number
$$\nu_\ty:=\dfrac{h_1}{e_1}+\dfrac{h_2}{e_1e_2}+\cdots+\dfrac{h_r}{e_1\cdots e_r}.
$$

By (iv) of Proposition \ref{onesided}, since $R_r(F)(y)\sim \psi_r(y)$, the Newton polygon $N_{r+1}^-(F)$ has length one and slope $\lambda_{r+1}=-h_{r+1}$, for some positive integer $h_{r+1}$; in particular, the minimal positive denominator of $\lambda_{r+1}$ is $e_{r+1}=1$.

\begin{lem}\label{firstappr}
Let $\phi(x)\in\zpx$ be a Montes approximation to $F(x)$, cons\-tructed as a representative of an $F$-complete optimal type $\ty$. Then, 
$$F(x)\equiv \phi(x)\md{\m^{\lceil\nu\rceil}}, \quad \mbox{ where }\nu=\nu_\ty+(h_{r+1}/e(L/K)).
$$
\end{lem}

\begin{proof}
Let $g(x)=\phi(x)-F(x)\in\zpx$. This polynomial has degree less than $n$. Take $e_0=1$; since,
$$
n-1=(e_rf_r-1)m_r+\cdots+(e_1f_1-1)m_1+(e_0f_0-1),
$$
Theorem \ref{allm} shows that
\begin{align*}
v(g(\t))-v(g(x))\le &\ v\left(\phi_r(\t)^{e_rf_r-1}\cdots \phi_1(\t)^{e_1f_1-1}\t^{e_0f_0-1}\right)\\ =&\ \sum_{i=1}^r(e_if_i-1)v(\phi_i(\t)). 
\end{align*}
Now, $g(\t)=\phi(\t)$, and by Lemma \ref{aux2}, applied to $\phi_{r+1}(x):=\phi(x)$, we get
\begin{align*}
v(g(x))\, \ge &\ v(\phi_{r+1}(\t))- \sum_{i=1}^r(e_if_i-1)v(\phi_i(\t))\\= &\ v(\phi_{r+1}(\t))- \sum_{i=1}^rv(\phi_{i+1}(\t))-v(\phi_i(\t))-\dfrac{h_{i+1}}{e_1\cdots e_{i+1}}\\ = &\ v(\phi_1(\t))+\sum_{i=1}^r\dfrac{h_{i+1}}{e_1\cdots e_{i+1}}=\nu. 
\end{align*}
\end{proof}

\subsection{Reinterpretation of Montes algorithm}\label{reinterpr}
Let $f(x)\in\zpx$ be a monic se\-parable polynomial.
Montes\ algorithm starts by computing the order zero types determined by the irreducible factors of $f(x)$ modu\-lo $\m$, and then proceeds to enlarge them in a convenient way till several $f$-complete optimal types $\ty_1,\dots,\ty_s$ are obtained, which are in bijective correspondence with the irreducible factors $f_{\ty_1}(x),\dots, f_{\ty_s}(x)$ of $f(x)$ in $\zpx$  \cite{HN}, \cite{GMN}. This one-to-one correspondence is determined by the following properties:
\begin{enumerate}
 \item For all $1\le i\le s$, the type $\ty_i$ is $f_{\ty_i}$-complete.
\item For all $j\ne i$, the type $\ty_j$ does not divide $f_{\ty_i}(x)$.
\end{enumerate}

This enlargement process of the types is based on a branching pheno\-menon, plus  a procedure taking care that all types we construct are optimal.
Let us briefly explain how this works. Suppose the algorithm considers an optimal type of order $i$
$$\ty=(\phi_1(x);\lambda_1,\phi_2(x);\cdots;\lambda_{i-1},\phi_i(x);\lambda_i,\psi_i(y))$$dividing $f(x)$. If this type is not $f$-complete, it may ramify to produce new types, that are the germs of different $f$-complete types (thus, of different irreducible factors). A representative $\phi_{i+1}(x)\in\zpx$ of $\ty$ is constructed (in a non-canonical way); then $N_{i+1}(f)=N_{\phi_{i+1},v_{i+1}}(f)$ is computed, and for each side of negative slope (say) $\lambda$, the residual polynomial 
$R_{\lambda}(f)(y)=R_{\phi_{i+1},v_{i+1},\lambda}(f)(y)\in\ff{i+1}[y]$ is computed and factorized into a product of irreducible factors. The type $\ty$ ramifies in principle into as many types as pairs $(\lambda,\psi(y))$, where $\lambda$ runs on the negative slopes of $N_{i+1}(f)$ and $\psi(y)$ runs on the different irreducible factors of $R_{\lambda}(f)(y)$. These branches determine either types of order $i+1$, or types of order $i$. To decide which is the case, one looks at the pair $(e_i,f_i)$, where $e_i$ is the least positive denominator of $\lambda_i$, and $f_i=\deg(\psi_i(y))$; if $e_if_i>1$ then $\phi_i(x)$ is an optimal representative of the truncated type $\ty_{i-1}$ (in the sense of Theorem \ref{optimal}), and the branches lead to new optimal types of order $i+1$
$$
\ty'=(\phi_1(x);\lambda_1,\phi_2(x);\cdots;\lambda_{i-1},\phi_i(x);\lambda_i,\phi_{i+1}(x);\lambda,\psi(y)),
$$dividing $f(x)$. If $e_if_i=1$ then $\phi_i(x)$ is not optimal, and we replace it by $\phi_i'(x):=\phi_{i+1}(x)$ to get a better representative of the truncated type $\ty_{i-1}$; this is called a \emph{refinement step}  \cite[Sect.3.2]{GMN}. In this latter case, we consider new optimal types of order $i$ dividing $f(x)$
$$
\ty'=(\phi_1(x);\lambda_1,\phi_2(x);\cdots;\lambda_{i-1},\phi'_i(x);\lambda,\psi(y)),\quad|\lambda|>|\lambda_i|,
$$ that will be analyzed and ramified in a similar way.

The final output of the algorithm is a list of $f$-complete optimal types. Therefore, Theorems \ref{main} and \ref{quasimain} have the following interpretation.

\begin{cor}\label{mainmain}
The output of Montes\ algorithm is a family of Okutsu frames and Montes approximations to all the irreducible factors of $f(x)$.
\end{cor}

Corollary \ref{mainmain} opens the door to new perspectives in the applications of Montes algorithm, as a tool to compute the essential arithmetic information about the irreducible factors of $f(x)$, that their Okutsu frames and their Montes approximations carry on \cite{GMN3}. 

We finish this section with an example that illustrates when strongly optimal $f$-complete types may occur. \medskip

\noindent{\bf Example. }Let $p$ be an odd prime number, and $\phi(x)\in\Z_p[x]$ a monic polynomial which is irreducible modulo $p$. Take $c\in\Z_p$ such that $c\equiv1\md{p}$, and consider the polynomial $f(x)=\phi(x)^2+p^2c$.
Clearly, $N_\phi(f)$ is one-sided of length $2$ and slope $-1$, whereas $R_{\phi,-1}(f)(y)=y^2+1\in\ff{p}[y]$.

If $p\equiv3\md4$, Montes algorithm outputs the $f$-complete and strongly optimal type $\ty=(\phi(x);-1,y^2+1)$. In this case, $f(x)=f_\ty(x)$ is irreducible of depth one, $[\phi]$ is an Okutsu frame of $f(x)$, and any representative of $\ty$ (for instance, $\phi_2(x)=\phi(x)^2+p^2$) is a Montes approximation to $f(x)$.

If $p\equiv1\md4$, there is some $i\in\Z_p$ satisfying $i^2=-1$. In this case the residual polynomial $R_{\phi,-1}(f)(y)$ factorizes as $y^2+1=(y-\overline{i})(y+\overline{i})$ in $\ff{p}[y]$, and Montes algorithm outputs two types:
$$
\ty=(\phi(x);-1,y-\overline{i});\quad \ty'=(\phi(x);-1,y+\overline{i}). 
$$ 
These types are $f$-complete and optimal, but not strongly optimal;  they correspond to the two irreducible factors of $f(x)$:  $f_\ty(x)=\phi(x)-ip\sqrt{c}$, $f_{\ty'}(x)=\phi(x)+ip\sqrt{c}$, both of depth zero. Let $\phi_2(x)=\phi(x)-ip$, $\phi_2'(x)=\phi(x)+ip$, be representatives respectively of $\ty$ and $\ty'$. The five polynomials $\phi,f_\ty,f_{\ty'},\phi_2,\phi_2'$ are one a Montes approximation to each other, since
$$
f_\ty(x)\equiv f_{\ty'}(x)\equiv \phi_2(x)\equiv \phi_2'(x)\equiv\phi(x)\md{p}. 
$$  
They are all representatives of the type of order zero, $\ty_0=\phi(y)\md{p}$, which is $f_\ty$-complete and $f_{\ty'}$-com\-plete, but not $f·$-complete.

Note that $\phi_2(x)$ (respectively $\phi_2'(x)$) is a better approximation to $f_\ty(x)$
(respectively  $f_{\ty'}(x)$) than $\phi(x)$. This is the basic idea behind a method to get approximations of arbitrarily high precision, to be explained in the next section.

\subsection{Applications to local factorization}\label{factorization}
For many purposes, a Montes approxi\-mation yields sufficient arithmetic information about an irreducible factor $F(x)$ of $f(x)$, the extension $L/K$ that it determines, and its subextensions. However, the setting of Montes algorithm can also be used to compute an approximation to each irreducible factor of $f(x)$, with an arbitrary prescribed precision, leading in this way to a factorization algorithm. The basic idea is to continue the process of enlarging the $f$-complete types.

Let us describe this factorization algorithm in detail. The input polynomial is a monic separable polynomial $f(x)\in\zpx$. First, we compute all $f$-complete types parameterizing the irreducible factors of $f(x)$, and a Montes approximation to each factor, by a single call to Montes\ algorithm. Then, for each complete type $\ty$, with representative $\phi_{r+1}(x)$, we compute 
the Newton polygon $N_{r+1}^-(f)$ and the residual polynomial $R_{r+1}(f)(y)\sim \psi_{r+1}(y)$. By Proposition \ref{onesided}, this polygon has length one, a negative slope $\lambda_{r+1}=-h_{r+1}$, for some positive integer $h_{r+1}$, and $\psi_{r+1}(y)\in\ff{r+1}[y]$ is a monic irreducible polynomial of degree one. 

In order to get a better approximation to $f_\ty(x)$, we compute a representative $\phi_{r+2}(x)$
of the type of order $r+1$:
$$
\ty'=(\phi_1(x);\lambda_1,\phi_2(x);\cdots;\lambda_{r-1},\phi_r(x);\lambda_r,\phi_{r+1}(x);-h_{r+1},\psi_{r+1}(y)).
$$
By Lemma \ref{aux2}, $v(\phi_{r+2}(\t))>v(\phi_{r+1}(\t))$, so that $\phi_{r+2}(x)$ is another Montes approximation to $f_\ty(x)$, by (ii) of Lemma-Definition \ref{Mapp}. The slope of the polygon $N_{r+2}(f)$ determines the precision of the new approximation. However, computations in order $r+2$ have a higher complexity than computations in order $r+1$. Therefore, we consider $\phi_{r+2}(x)$ as a new representative of the type $\ty$, say $\phi'_{r+1}(x):=\phi_{r+2}(x)$, and we repeat the procedure with the new representative: we compute the slope $-h'_{r+1}$ of $N'_{r+1}(f)=N'_{\phi'_{r+1},v_{r+1}}(f)$ and the residual polynomial $R'_{-h_{r+1}}(f)(y)$. By Lemma \ref{firstappr}, the new precision of the approximation is 
$$
\nu'=\nu_\ty+\dfrac{h'_{r+1}}{e(L/\Q_p)}.
$$
Since $v(\phi'_{r+1}(\t))>v(\phi_{r+1}(\t))$, we have $h'_{r+1}>h_{r+1}$, and the new approximation is better than the former one. Note that each iteration requires three computations in order $r+1$: a Newton polygon of length one, a residual polynomial of degree one and a representative of the extended type $\ty'$. One can perform the necessary number of iterations till a prefixed precision is attained.

The scheme of the factorization algorithm would be the following:\medskip

\noindent{INPUT. }A monic separable polynomial $f(x)\in\zpx$ and a precision $N\in\N$.
\vskip.2cm

\noindent{OUTPUT. }A family of monic irreducible polynomials $f_1(x),\dots, f_s(x)$  in $\zpx$,  satisfying $f_i(x)\equiv f_{\ty_i}(x)\md{\m^N}$, for $1\le i\le s$, where  $f_{\ty_1}(x),\dots f_{\ty_s}(x)$ are the genuine irreducible factors of $f(x)$ in $\zpx$.   
\vskip.2cm

\noindent{\bf 1. } Apply Montes\ algorithm to compute $f$-complete types $\ty_1,\dots,\ty_s$, corres\-ponding to the irreducible factors of $f(x)$ in $\zpx$.

\noindent{\bf 2. } For each $f$-complete type $\ty$, let $r$ be the order of $\ty$, and:

\qquad {\bf 3. } Compute $\nu_\ty=\sum_{i=1}^rh_i/(e_1\cdots e_i)$, and $e=e_1\cdots e_r$.

\qquad {\bf 4. } Compute a representative $\phi(x)\in\zpx$ of $\ty$. 

\qquad {\bf 5. } Compute $N_{r+1}(f)$, and let the slope of this one-sided polygon be $-h$. If $\lceil \nu_\ty+(h/e)\rceil\ge N$
output $\phi(x)$ as the final approximation to $f_\ty(x)$, and consider the next $f$-complete type. Else: 

\qquad {\bf 6. } Compute $R_{r+1}(f)(y)\sim \psi(y)$, for some monic polynomial $\psi(y)\in\ff{r+1}[y]$ of degree one. 

\qquad {\bf 7. } Compute a representative $\phi'(x)$ of the type $\ty'$ of order $r+1$ obtained by enlarging $\ty$ with the triple $(\phi;-h,\psi(y))$.

\qquad {\bf 8. } Replace $\phi(x)\leftarrow \phi'(x)$ and go to step {\bf 5}.\medskip

This algorithm has some advantages: it has very low memory requirements, and it always outputs a family $f_1(x),\dots,f_s(x)\in\zpx$ of monic irreducible polynomials that satisfy $f(x)\equiv f_1(x)\cdots f_s(x)\md{\m^N}$, regardless of the value of the imposed precision $N$.

As to the disadvantages: although the algorithm computes very fastly each new approximation, it has a slow convergence because at each step the improvement of the precision is rather small. 

Probably, an optimal local factorization algorithm would consist in the application of Montes algorithm as a fast method to get a Montes approxi\-mation to each irreducible factor, combined with an efficient ``Hensel lift" routine able to improve these initial approximations by doubling the precision at each iteration. One may speculate if Newton polygons of higher order might be used too to design a similar acceleration procedure.

\subsection*{Acknowledgements}

Partially supported by MTM2009-13060-C02-02  and \linebreak MTM2009-10359 from the Spanish MEC.

\end{document}